\documentclass[11pt]{amsart}
\usepackage{a4wide}
\usepackage[latin1]{inputenc}
\usepackage{amsmath}
\usepackage{amsfonts}
\usepackage{amssymb}
\usepackage{graphicx}
\usepackage{amsthm}
\usepackage{amssymb}
\usepackage{cite}
\usepackage{mathrsfs}
\usepackage{hyperref}
\usepackage{paralist}
\usepackage{enumitem}
\usepackage{tikz}
\usetikzlibrary{matrix,chains}
\makeindex
\begin{document}
	
	\newcommand{\N}{\mathbf{N}}
	\newcommand{\n}{\mathbf{n}}
	\newcommand{\x}{\mathbf{x}}
	\newcommand{\h}{\mathbf{h}}
	\newcommand{\m}{\mathbf{m}}
	
	\newcommand{\B}{\mathbf{B}}
	\newcommand{\U}{\mathbf{U}}
	\newcommand{\V}{\mathbf{V}}
	\newcommand{\T}{\mathbf{T}}
	\newcommand{\G}{\mathbf{G}}
	\newcommand{\Para}{\mathbf{P}}
	\newcommand{\Levi}{\mathbf{L}}
	\newcommand{\Y}{\mathbf{Y}}
	\newcommand{\X}{\mathbf{X}}
	\newcommand{\M}{\mathbf{M}}
	\newcommand{\pro}{\mathbf{prod}}
	\renewcommand{\o}{\overline}
	
	\newcommand{\Gtilde}{\mathbf{\tilde{G}}}
	\newcommand{\Ttilde}{\mathbf{\tilde{T}}}
	\newcommand{\Btilde}{\mathbf{\tilde{B}}}
	\newcommand{\Ltilde}{\mathbf{\tilde{L}}}
	\newcommand{\C}{\operatorname{C}}
	
	\newcommand{\bl}{\operatorname{bl}}
	\newcommand{\Z}{\operatorname{Z}}
	\newcommand{\Gal}{\operatorname{Gal}}
	\newcommand{\kernel}{\operatorname{ker}}
	\newcommand{\Irr}{\operatorname{Irr}}
	\newcommand{\D}{\operatorname{D}}
	\newcommand{\I}{\operatorname{I}}
	\newcommand{\GL}{\operatorname{GL}}
	\newcommand{\SL}{\operatorname{SL}}
	\newcommand{\W}{\operatorname{W}}
	\newcommand{\R}{\operatorname{R}}
	\newcommand{\Br}{\operatorname{Br}}
	\newcommand{\Aut}{\operatorname{Aut}}
	\newcommand{\End}{\operatorname{End}}
	\newcommand{\Ind}{\operatorname{Ind}}
	\newcommand{\Res}{\operatorname{Res}}
	\newcommand{\br}{\operatorname{br}}
	\newcommand{\Hom}{\operatorname{Hom}}
	\newcommand{\Endo}{\operatorname{End}}
	\newcommand{\Ho}{\operatorname{H}}
	\newcommand{\Tr}{\operatorname{Tr}}
	\newcommand{\opp}{\operatorname{opp}}
	\theoremstyle{remark}

	\theoremstyle{definition}
	\newtheorem{definition}{Definition}[section]
	\newtheorem{notation}[definition]{Notation}
	\newtheorem{construction}[definition]{Construction}
	\newtheorem{remark}[definition]{Remark}
	\newtheorem{example}[definition]{Example}

	\theoremstyle{plain}
	
	\newtheorem{theorem}[definition]{Theorem}
	\newtheorem{lemma}[definition]{Lemma}
	\newtheorem{question}{Question}
	\newtheorem{corollary}[definition]{Corollary}
	\newtheorem{proposition}[definition]{Proposition}
	\newtheorem{conjecture}[definition]{Conjecture}
	\newtheorem{assumption}[definition]{Assumption}
	\newtheorem{hypothesis}[definition]{Hypothesis}
	\newtheorem{maintheorem}[definition]{Main Theorem}
	
	\newtheorem*{hypo*}{Hypothesis 3.3'}

	\newtheorem*{theo*}{Theorem}
	\newtheorem*{conj*}{Conjecture}
	\newtheorem*{cor*}{Corollary}
	
	\newtheorem{theo}{Theorem}
	\newtheorem{conj}[theo]{Conjecture}
	\newtheorem{cor}[theo]{Corollary}

	\renewcommand{\thetheo}{\Alph{theo}}
	\renewcommand{\theconj}{\Alph{conj}}
	\renewcommand{\thecor}{\Alph{cor}}

	\title{Jordan Decomposition for the Alperin--McKay Conjecture}
	
	\date{\today}
	\author{Lucas Ruhstorfer}
	\address{Fachbereich Mathematik, TU Kaiserslautern, 67653 Kaiserslautern, Germany}
	\email{ruhstorfer@mathematik.uni-kl.de}

	
	\begin{abstract}
		Späth showed that the Alperin--McKay conjecture in the representation theory of finite groups holds if the so-called inductive Alperin--McKay condition holds for all finite simple groups. In a previous article, we showed that the Bonnafé--Rouquier equivalence for blocks of finite groups of Lie type can be lifted to include automorphisms of groups of Lie type. We use our results to reduce the verification of the inductive condition for groups of Lie type to quasi-isolated blocks.
	\end{abstract}

	\maketitle
	
	\section*{Introduction}
	
	\subsection*{The Global-Local Conjectures}
	
	In the representation theory of finite groups one of the most important conjectures is the so-called Alperin--McKay conjecture. Let $G$ be a finite group and $\ell$ a prime dividing the order of $G$. For an $\ell$-block $b$ of $G$ we denote by $\Irr_0(G,b)$ the set of height zero characters of $b$. Then the Alperin--McKay conjecture is as follows.

	\begin{conj*}[Alperin--McKay]
		Let $b$ be an $\ell$-block of $G$ with defect group $D$ and $B$ its Brauer correspondent in $\mathrm{N}_G(D)$. Then $$|\Irr_0(G,b)|= |\Irr_0(\mathrm{N}_G(D),B)|.$$
	\end{conj*}
	
	This conjecture has been reduced by Späth \cite[Theorem C]{IAM} to the verification of certain stronger versions of the same conjecture for finite quasi-simple groups. These stronger versions are usually referred to as the inductive conditions. According to the classification of finite simple groups, many finite simple groups are groups of Lie type. Therefore, understanding their representation theory is key to proving the inductive conditions for simple groups.
	
	\subsection*{Representation theory of groups of Lie type}

In the representation theory of groups of Lie type the Morita equivalences by Bonnafé--Rouquier have proved to be an indispensable tool. Let $\G$ be a connected reductive group defined over an algebraic closure of $\mathbb{F}_p$ with Frobenius endomorphism $F: \G \to \G$. Suppose that $(\G^\ast,F^\ast)$ is a group in duality with $(\G,F)$ and fix a semisimple element $ s\in (\G^\ast)^{F^\ast}$ of $\ell'$-order. We assume that $\ell$ is a prime with $\ell \neq p$ and we fix an $\ell$-modular system as in \ref{iAM} below and denote by $e_s^{\G^F} \in \mathrm{Z}(\mathcal{O} \G^F)$ the central idempotent associated to the $(\G^\ast)^{F^\ast}$-conjugacy class of $s$ as in Brou\'e--Michel \cite{BrMi}. Suppose that $\Levi^\ast$ is the minimal $F^\ast$-stable Levi subgroup of $\G^\ast$ containing $\mathrm{C}_{(\G^\ast)^{F^\ast}}(s)\mathrm{C}_{\G^\ast}^\circ(s)$. Let $\Para$ be a parabolic subgroup of $\G$ with Levi decomposition $\Para= \Levi \ltimes \U$ and let $\Y_\U^\G$ the Deligne--Lusztig variety associated to it.
%	Then the associated \textit{Deligne--Lusztig} variety
%	$$\Y_\U^\G= \{ g \U \in \G/ \U \mid g^{-1} F(g) \in \U F(\U) \}$$
%	has a left $\G^F$- and a right $\Levi^F$-action. Its $\ell$-adic cohomology groups $H^i_c(\Y_\U^\G,\mathcal{O})$ can therefore be considered as $\Lambda \G^F$-$\Lambda \Levi^F$-bimodules.
	Then the following was proved in \cite{BoRo}.

	\begin{theo}[Bonnaf\'e--Rouquier]\label{BDRintro}
With the notation as above, the Deligne--Lusztig bimodule $H_c^{\operatorname{dim}(\Y_{\U})}(\Y_\U^\G,\mathcal{O}) e_s^{\Levi^F}$ induces a Morita equivalence between $\mathcal{O} \G^F e_s^{\G^F}$ and  $\mathcal{O} \Levi^F e_s^{\Levi^F}$.
	\end{theo} 
	
	%\begin{theo}\label{finalintro}
	%Suppose that $\G$ is a simple algebraic group and denote by $\mathbf{N}^F$ the stabilizer of $\Levi$ and $e_s^{\Levi^F}$ in $\G^F$. If $ \ell \nmid (q^2-1)$ or if $\mathbf{N}^F/ \Levi^F$ is cyclic then the bimodule $H_c^{\mathrm{dim}(\Y_\U^\G)}(\Y_\U^\G,\mathcal{O}) e_s^{\Levi^F}$ extends to a $\mathcal{O} \G^F$-$\mathcal{O} \N^F $-bimodule such that $\mathcal{O} \G^F e_s^{\G^F}$ and $\mathcal{O} \N^F e_s^{\Levi^F}$ are Morita equivalent. Moreover, $\mathcal{O} \G^F e_s^{\G^F}$ and $\mathcal{O} \N^F e_s^{\Levi^F}$ are splendid Rickard equivalent.
	%\end{theo}
	
	%Note that Theorem \ref{finalintro} has also appeared in the author's article \cite{Lucas}.
	This result was substantially improved by Bonnafé--Dat--Rouquier \cite{Dat}. For instance they prove that the algebras $\mathcal{O} \Levi^F e_s^{\Levi^F}$ and $\mathcal{O} \G^F e_s^{\G^F}$ are even splendid Rickard equivalent. This for instance implies that the Brauer categories of corresponding blocks are equivalent.
	
	\subsection*{Clifford theory and group automorphisms}
	
%	The Jordan decomposition by Bonnaf\'e--Rouquier has proved to be extremely useful in the representation theory of finite groups of Lie type. For instance, the Bonnaf\'e--Rouquier Morita equivalence was a crucial ingredient in the verification of one direction of Brauer's height zero conjecture by Malle--Kessar \cite{KessarMalle}. Our main objective in this article is to provide a reduction of the verification of the inductive Alperin--McKay condition from SpÃ¤th \cite{IAM} to blocks associated to quasi-isolated semisimple elements. 
	In \cite{Jordan}, we extended the Bonnafé--Rouquier equivalences by including automorphisms of the finite group $\G^F$.
	Let $\G$ be a simple algebraic group of simply connected type (not of type $D_4$) such that $\G^F/\mathrm{Z}(\G^F)$ is a finite simple group. We let $ \iota: \G \hookrightarrow \Gtilde$ be a regular embedding. In \cite{Jordan} we showed the existence of a finite abelian group $\mathcal{A} \subseteq \mathrm{Aut}(\Gtilde^F)$ which together with the automorphisms induced by $\Gtilde^F$ generates the stabilizer of $e_s^{\G^F}$ in $\mathrm{Out}(\G^F)$. Denote by $\Delta \mathcal{A}$ the diagonal embedding of $\mathcal{A}$ in $\G^F \mathcal{A} \times (\Levi^F \mathcal{A})^{\mathrm{opp}}$.
	%As before, consider a semisimple element $ s\in (\G^\ast)^{F^\ast}$ of $\ell'$-order and suppose now that $\Levi^\ast$ is the minimal $F^\ast$-stable Levi subgroup of $\G^\ast$ with $\mathrm{C}_{\G^\ast}(s) \subseteq \Levi^\ast$.

%Moreover, these bijective morphisms commute with each other and the Frobenius endomorphism $F: \G \to \G$ is an integral power of $F_0$. Using this explicit description of automorphisms we can prove the following:
	
	%\begin{proposition}\label{intro1}
	% $\mathcal{A}=\langle F_0, \sigma \rangle \subseteq \mathrm{Aut}(\tilde{\G}^F)$ satisfies:
	%\begin{enumerate}[label=(\alph*)]
	%\item $F_0 \circ \sigma = \sigma \circ F_0$ as morphisms of $\tilde{\G}$.
	%\item 
	%\item There exists a Levi subgroup $\Levi$ in duality with $\Levi^\ast$ such that $\mathcal{A}$ stabilizes $\Levi$ and $e_s^{\Levi^F}$.
	%\end{enumerate}
	%\end{proposition}
	%The automorphism $F_0: \G \to \G$ can be thought of as a field automorphism inducing an $\mathbb{F}_{q'}$-structure such that $F_0^r=F$.
	%The automorphism $\sigma: \G \to \G$ is an automorphism which comes from the structure analysis of the stabilizer of $e_s^{\G^F}$ and has to be chosen very carefully. 
	%SHORTER
	\begin{theo}\label{intro2}
		Assume that $|\mathcal{A}|$ is coprime to $\ell$.
		Then $H^{\mathrm{dim}(\Y_\U^\G)}_c(\Y_\U^\G,\mathcal{O}) e_s^{\Levi^F}$ extends to an $\mathcal{O}[ (\G^F \times (\Levi^F)^{\mathrm{opp}}) \Delta(\mathcal{A})]$-module $M$. Moreover, the bimodule $\mathrm{Ind}_{ (\G^F \times (\Levi^F)^{\mathrm{opp}}) \Delta(\mathcal{A}) }^{\tilde{\G}^F \mathcal{A} \times (\tilde{\Levi}^F \mathcal{A})^{\mathrm{opp}}} (M)$ induces a Morita equivalence between $\mathcal{O} \tilde{\Levi}^F \mathcal{A} e_s^{\Levi^F}$ and $\mathcal{O} \tilde{\G}^F \mathcal{A} e_s^{\G^F}$.
	\end{theo}

This theorem most importantly shows that the Bonnafé--Rouquier Morita equivalences preserve the Clifford theory with respect to automorphisms coming from $\mathcal{A}$. Moreover, using the splendid Rickard equivalence we were also able to provide a local version of Theorem \ref{intro2}.

	\subsection*{Jordan decomposition for the Alperin--McKay conjecture}
	
Our aim in this article is to show that the strong equivariance properties obtained in Theorem \ref{intro2} and its local version can be used to give a reduction of the inductive Alperin--McKay condition to quasi-isolated blocks. 
%	This gives us reason to hope that our reduction will provide a simplification of the verification of the inductive Alperin--McKay conditions.
	Our main theorem is then as follows:

	\begin{theo}[see Theorem \ref{maintheorem}]\label{intromain}
		Assume that every quasi-isolated $\ell$-block of a finite quasi-simple group of Lie type defined over a field of characteristic $p \neq \ell$ satisfies the inductive Alperin--McKay condition (in the sense of Hypothesis \ref{assumptionquasi}' below). Let $S$ be a simple group of Lie type  with non-exceptional Schur multiplier defined over a field of characteristic $p \neq \ell$ and $G$ its universal covering group. Then the inductive Alperin--McKay condition holds for every $\ell$-block of $G$.
	\end{theo}

%	Note that Assumption \ref{starcondition} was one of the main ingredients for the previous verifications of the inductive McKay condition. It is essential for constructing projective representations for certain classes of groups associated to characters of groups of Lie type and enables us to explicitly compute the factor set of the projective representation. The proof of Assumption \ref{starcondition} for groups of type $D$ is addressed in current work of Späth. Hence, Theorem \ref{intromain} is expected to yield a complete reduction of the verification of the inductive Alperin--McKay condition to quasi-isolated blocks.

	Quasi-isolated semisimple elements for reductive groups have been classified by Bonnaf\'e \cite{Bonnafe}. In each case there are a small number of possibilities which have been well-described. Moreover, the quasi-isolated blocks of groups of Lie type are better understood by fundamental work of Cabanes--Enguehard and recent work of Enguehard and Kessar--Malle, see \cite{KessarMalle} for a more precise historical account. Therefore, our reduction is a first step towards understanding the inductive condition for these groups and we hope that it will provide significant simplifications in verifying it.

	%% SOMETHING ABOUT BOLTJE PEREPELITSKY AND THE ALPERIN WEIGHT CONJECTURE.

	%Note that a splendid equivalence induces a Rickard equivalence between corresponding Brauer pairs on the level of centralizers. One can then prove (using the fact that the Brauer categories are equivalent) that this lifts to an equivalence between the normalizers. However, note that we did not claim in Theorem \ref{intro2} that the equivalence between $\mathcal{O} \tilde{\Levi}^F \mathcal{A} e_s^{\Levi^F}$ and $\mathcal{O} \tilde{\G}^F \mathcal{A} e_s^{\G^F}$ is splendid. To remedy this we directly use the local equivalences established by Bonnaf\'e--Dat--Rouquier and show that they lift (in the same spirit as in Theorem \ref{intro2}).
	\subsection*{Acknowledgement}
	%
	%Firstly,
	
	%Finally, Britta Sp\"{a}th for her continuous support. I would like to thank the referee for their helpful commentary on a previous version of this paper.
	%
	%
The results of this article were obtained as part of my Phd thesis at the Bergische Universität Wuppertal. Therefore, I would like to express my gratitude to my supervisor Britta Späth for suggesting this topic and for her constant support. I would like to thank Marc Cabanes for reading through my thesis and for many interesting discussions. I am deeply indebted to Gunter Malle for his suggestions and his thorough reading of an earlier version of this paper. I would also like to thank Damiano Rossi for pointing out an error in a previous version of this paper.
	\ \\
	This material is partly based upon work supported by the NSF under Grant DMS-1440140 while the author was in residence at the MSRI, Berkeley CA. The research was conducted in the framework of the research training group GRK 2240: Algebro-geometric
	Methods in Algebra, Arithmetic and Topology, which is funded by the DFG.

\section{Application to the inductive Alperin--McKay condition}\label{Applications}

%In this chapter we show how the results from the previous section can be used in the verification of the Alperin--McKay conjecture. More precisely, we show that in order to prove the inductive Alperin-McKay condition for all blocks of groups of Lie type it is sufficient to consider their (strictly) quasi-isolated blocks.
%\newpage
%
%
%
%
%\newpage
%\subsection{Application to local global conjecture}
%
%In this section we discuss the application of our results to the Alperin--McKay conjecture.
\subsection{The inductive Alperin--McKay condition}\label{iAM}

The aim of this section is to recall the inductive Alperin--McKay condition as introduced in \cite[Definition 7.2]{IAM}. Let $\ell$ be a prime and $K$ be a finite field extension of $\mathbb{Q}_\ell$. We say that $K$ is \textit{large enough for a finite group} $G$ if $K$ contains all roots of unity whose order divides the exponent of the group $G$. In the following, $K$ denotes a field which we assume to be large enough for the finite groups under consideration. We denote by $\mathcal{O}$ the ring of integers of $K$ over $\mathbb{Z}_\ell$ and by $k=\mathcal{O}/J(\mathcal{O})$ its residue field. We will use $\Lambda$ to interchangeably denote $\mathcal{O}$ or $k$. Moreover, $\Irr(G)$ will denote the set of $K$-valued irreducible characters of $G$.

Recall that a \textit{character triple} $(G,N,\theta)$ consists of a finite group $G$ with normal subgroup $N$ and a $G$-invariant character $\theta \in \Irr(N)$. A \textit{projective representation} is a set-theoretic map $\mathcal{P}: G \to \mathrm{GL}_{n}(K)$ such that for all $g, g' \in G$ there exists a scalar $\alpha(g,g') \in K$ with $\mathcal{P}(g g')= \alpha(g,g') \mathcal{P}(g) \mathcal{P}(g')$. The projective representation $\mathcal{P}: G \to \mathrm{GL}_{n}(K)$ is said to be \textit{associated to $(G,N, \theta)$} if the restriction $\mathcal{P}|_N$ affords the character $\theta$ and for all $n \in N$ and $g \in G$ we have $\mathcal{P}(gn)=\mathcal{P}(g) \mathcal{P}(n)$ and $\mathcal{P}(ng)=\mathcal{P}(n) \mathcal{P}(g)$.

We recall the following order relation on character triples, see \cite[Definition 2.1]{LocalRep}:

\begin{definition}
	Let $(G,N,\theta)$ and $(H,M, \theta')$ be two character triples. We write
	$$(G,N,\theta) \geq (H,M,\theta')$$
	if the following conditions are satisfied:
	\begin{enumerate}[label=(\roman*)]
		\item $G=NH$, $M= N \cap H$ and $\mathrm{C}_G(N) \leq H$.
		\item There exist projective representations $\mathcal{P}$ and $\mathcal{P}'$ associated with $(G,N, \theta)$ and $(H,M, \theta')$ such that their factor sets $\alpha $ and $\alpha'$ satisfy $\alpha|_{H \times H}= \alpha'$.
	\end{enumerate}
\end{definition}

Let $\mathcal{P}: G \to \mathrm{GL}_{n}(K)$ be a projective representation associated to the character triple $(G,N, \theta)$. Then for $c \in \mathrm{C}_G(N)$ we have $\mathcal{P}(c) \mathcal{P}(n)= \mathcal{P}(n) \mathcal{P}(c)$ for every $n \in N$. Since $\mathcal{P}|_N$ affords the irreducible character $\theta$ it follows by Schur's lemma that $\mathcal{P}(c)$ is a scalar matrix.

For the inductive Alperin--McKay condition we need a refinement of the order relation ``$\geq$'' on character triples. More precisely, we also require that the scalars of the projective representations on elements of $\mathrm{C}_G(N)$ coincide:

\begin{definition}
	Let $(G,N,\theta)$ and $(H,M, \theta')$ be two character triples such that $(G,N,\theta) \geq (H,M,\theta')$ via the projective representations $\mathcal{P}$ and $\mathcal{P'}$. Then we write
	$$(G,N,\theta)  \geq_c(H,M,\theta')$$ if for every $c \in \mathrm{C}_G(N)$ the scalars associated to $\mathcal{P}(c)$ and $\mathcal{P}'(c)$ coincide.
	%\begin{enumerate}[label=(\alph*)]
	%
	%\item We have such that:
	%%\item $\mathrm{C}_G(N) \leq H$.
	%\item 
	%\end{enumerate}
\end{definition}

%\begin{proposition}
%Let $(G,N,\theta)$ and $(H,M, \theta')$ be two character triples with $\theta \in \Irr(N)$ and $ \theta' \in \Irr(M)$. Then $(G,N,\theta) \geq_{c} (H,M,\theta')$ if and only if the following there exists projective representations $\mathcal{P}$ and $\mathcal{P}'$ of $ G$ and $H$ associated with $\theta$ and $\theta'$ such that
%\begin{enumerate}
%\item The factor sets of $\mathcal{P}$ and $\mathcal{P}'$ coincide on $H \times H$, and
%\item for every $x \in \mathrm{C}_G(N)$ the matrices $\mathcal{P}(x)$ and $\mathcal{P}'(x)$ are associated with the same scalar in $K$.
%\end{enumerate}
%\end{proposition}
%
%For the next definition we need to recall
Let $N$ be a normal subgroup of a finite group $G$ and $\chi \in \Irr(G)$. Then we write $\Irr(N \mid \chi)$ for the set of irreducible constituents of $\mathrm{Res}_N^G(\chi)$. Moreover for $\theta \in \Irr(N)$ we write $\Irr(G \mid \theta)$ for the set of all irreducible constituents $\chi$ of the induced character $\Ind_N^G(\theta)$. We then say that the character $\chi$ \textit{covers} the character $\theta$. 

\begin{theorem}\label{sigmabijections}
	Let $(G,N,\theta)$ and $(H,M, \theta')$ be two character triples such that $(G,N,\theta) \geq (H,M,\theta')$ with respect to the projective representations $\mathcal{P}$ and $\mathcal{P}'$. Then for every intermediate subgroup $N \leq J \leq G$ there exists a bijection
	$$\sigma_J: \mathbb{N} \Irr(J \mid \theta) \to \mathbb{N}\Irr(J \cap H \mid \theta')$$ such that $\sigma_J(\Irr(J \mid \theta))=\Irr(J \cap H \mid \theta')$.
\end{theorem}

\begin{proof}
	This is \cite[Theorem 2.2]{LocalRep}.
	%We sketch how this bijection is constructed.
	%Let $\chi \in \Irr(J \mid \theta)$. Then by \cite[Theorem 1.15]{LocalRep} there exists a projective representation $\mathcal{Q}$ of $J$ with factor set $\alpha^{-1}$ such that $\mathcal{D}= \mathcal{Q} \otimes \mathcal{P}_J$ affords the character $\chi$. Then we define $\sigma_J(\chi)$ to be the character afforded by the representation $\mathcal{Q}_{J \cap H} \otimes \mathcal{P'}_{J \cap H }$.
\end{proof}

The following properties which are collected in the next lemma are a consequence of the fact that ``$\geq$'' induces a strong isomorphism of character triples in the sense of \cite[Problem 11.13]{Isaacs}.
%We recall some known properties of character triples related by the relation $"\geq"$.

\begin{lemma}\label{propertiesofc}
	Let $(G,N,\theta)$ and $(H,M, \theta')$ be two character triples satisfying $(G,N,\theta) \geq (H,M,\theta')$.
	Then for any $N \leq J_1 \leq J_2 \leq G$ and $\chi \in \mathbb{N} \Irr(J_2 \mid \theta)$ the following holds:
	\begin{enumerate}[label=(\alph*)]
		\item $\mathrm{Res}^{J_2 \cap H}_{J_1 \cap H}(\sigma_{J_2}(\chi))= \sigma_{J_1}( \mathrm{Res}_{J_1}^{J_2}(\chi))$.
		\item $(\sigma_{J_2}(\chi \beta))= \sigma_{J_2}(\chi) \mathrm{Res}^{J_2}_{J_2 \cap H}(\beta)$ for every $\beta \in \mathbb{N} \Irr(J_2/N)$.
		\item $(\sigma_{J_2}(\chi))^{h}= \sigma_{J_2^h}(\chi^h)$ for every $h \in H$.
	\end{enumerate}
\end{lemma}

\begin{proof}
	This is \cite[Corollary 2.4]{LocalRep}.
\end{proof}

\begin{lemma}\label{propertiesofcpartd}
	Let $(G,N,\theta)$ and $(H,M, \theta')$ be two character triples with $(G,N,\theta) \geq_{c} (H,M,\theta')$. Then for every $N \leq J \leq G$ and $\kappa \in \Irr( \mathrm{C}_J(G))$ we have 
	$$\sigma_J( \Irr(J \mid \theta) \cap \Irr(J \mid \kappa)) \subseteq \Irr(J \cap H \mid \kappa).$$
\end{lemma}

\begin{proof}
	See \cite[Lemma 2.10]{LocalRep}.
\end{proof}

Let $G$ be a finite group and $\chi \in \Irr(G)$. Then we write $\mathrm{bl}(\chi)$ for the $\ell$-block of $G$ containing $ \chi$. The following definition is taken from \cite[Definition 4.2]{LocalRep}.
%
%The inductive conditions are phrased in terms of block induction, see \cite[p.87]{NavarroBook}. We will now recall this notion in the special case we are interested in, see \cite[Theorem 4.14]{NavarroBook}.

\begin{definition}\label{blockisomorphism}
	Let $(G,N,\theta)$ and $(H,M, \theta')$ be two character triples  with $(G,N,\theta) \geq_{c} (H,M,\theta')$. Then we write 
	$$(G,N,\theta) \geq_{b} (H,M,\theta')$$
	if the following hold:
	\begin{enumerate}[label=(\roman*)]
		\item A defect group $D$ of $\mathrm{bl}(\theta')$ satisfies $\mathrm{C}_G(D) \leq H$.
		\item The maps $\sigma_J$ induced by $(\mathcal{P},\mathcal{P'})$ satisfy
		$$\mathrm{bl}(\psi)= \mathrm{bl}(\sigma_J(\psi))^J$$
		for every $N \leq J \leq G$ and $\psi \in \Irr(J \mid \theta)$.
	\end{enumerate}
	
\end{definition}
%
%\begin{lemma}
%Let $(G,N,\theta)$ and $(H,M, \theta')$ two character triples. Assume that $G=NH$, $M= N \cap H$ and $\mathrm{C}_G(N) \leq H$. Denote $Z:= \mathrm{Z}(G) \cap \mathrm{Ker}(\theta)$ and let $\overline{\theta} \in \Irr(N/Z)$ and $\overline{\theta'} \in \Irr(M/Z)$ be the characters which lift to $\theta$ and $\theta'$ respectively. If $(G/Z,N/Z,\overline{\theta}) \geq_{b} (H/Z,M/Z,\overline{\theta'})$ then we have $(G,N,\theta) \geq_{b} (H,M,\theta')$
%\end{lemma}
%

\begin{remark}
	Let $G$ be a finite group and $b$ a block of $G$ with defect group $D$. Let $M$ be a subgroup of $G$ containing $\mathrm{N}_G(D)$. By Brauer correspondence there exists a unique block $B_D$ of $\mathrm{N}_G(D)$ such that $(B_D)^G=b$. On the other hand, since $\mathrm{N}_M(D)= \mathrm{N}_G(D)$ Brauer correspondence yields a bijection $\mathrm{Bl}(\mathrm{N}_G(D) \mid D) \to \mathrm{Bl}(M \mid D)$. It follows that $B:=(B_D)^M$ is the unique block of $M$ with defect group $D$ satisfying $B^G=b$, see \cite[Problem 4.2]{NavarroBook}.
\end{remark}

If $G$ is a finite group with normal subgroup $N$ and $\chi  \in \Irr(N)$ an irreducible character, then we write $G_\chi$ for the inertia group of the character $\chi$ in $G$.
%For $b$ a block of $G$ with defect group $D$ we denote by 
%$$\Irr_0(G,b):= \{ \chi \in \Irr(G,b) \mid \chi(1)_{\ell}=[G:D]_{\ell} \}$$
%the set of $\ell$\textit{-height zero characters} of the block $b$.

\begin{definition}\label{IAMsuitable}
	Let $G$ be a finite group and $b$ a block of $G$ with non-central defect group $D$. Assume that for $\Gamma:= \mathrm{N}_{\mathrm{Aut}(G)}(D,b)$ there exist
	\begin{enumerate}[label=(\roman*)]
		\item a $ \Gamma$-stable subgroup $M$ with $\mathrm{N}_G(D) \leq M \lneq G$;
		\item a $\Gamma$-equivariant bijection $\Psi: \Irr_0(G , b) \to \Irr_0(M , B)$ where $B \in \mathrm{Bl}(M \mid D)$ is the unique block with $B^G=b$; 
		\item $\Psi( \Irr_0(b  \mid \nu) ) \subseteq \Irr_0(B \mid \nu )$ for every $\nu \in \Irr(\mathrm{Z}(G))$ and
		$$( G \rtimes \Gamma_\chi, G, \chi) \geq_{b} ( M \rtimes \Gamma_\chi, M, \Psi(\chi) ),$$
		for every $ \chi \in \Irr_0(G,b).$
		%and $Z= \mathrm{Ker}(\chi) \cap \mathrm{Z}(G)$, where $\overline{\chi}$ and $\overline{\Psi(\chi)}$ lift to $\chi$ and $\Psi(\chi)$, respectively.
	\end{enumerate}
	Then we say that $\Psi: \Irr_0(G , b) \to \Irr_0(M , B)$ is an \textit{iAM-bijection for the block} $b$ with respect to the subgroup $M$.
\end{definition}

In the following, we will usually work with iAM-bijections. However, to formulate the inductive iAM-condition we need a slightly stronger version of Definition \ref{IAMsuitable}:

\begin{definition}
	We say that $\Psi: \Irr_0(G , b) \to \Irr_0(M , B)$ is a \textit{strong iAM-bijection for the block $b$} if it is an iAM-bijection which additionally satisfies
	$$( G/ Z \rtimes \Gamma_\chi, G/Z, \overline{\chi}) \geq_{b} ( M/Z \rtimes \Gamma_\chi, M/Z, \overline{\Psi(\chi)} ),$$
	for every $ \chi \in \Irr_0(G,b)$ and $Z= \mathrm{Ker}(\chi) \cap \mathrm{Z}(G)$, where $\overline{\chi}$ and $\overline{\Psi(\chi)}$ lift to $\chi$ and $\Psi(\chi)$, respectively.
\end{definition}
%
%\begin{remark}
Note that if $( G/ Z \rtimes \Gamma_\chi, G/Z, \overline{\chi}) \geq_{b} ( M/Z \rtimes \Gamma_\chi, M/Z, \overline{\Psi(\chi)} )$ then we automatically have $( G \rtimes \Gamma_\chi, G, \chi) \geq_{b} ( M \rtimes \Gamma_\chi, M, \Psi(\chi) )$ by \cite[Lemma 3.12]{JEMS}. The converse on the other hand is not necessarily true. We will however see that in most cases the difference between these two notions is not that big, see also Remark \ref{strong versus weak}.
%\end{remark}

Condition (iii) in Definition \ref{IAMsuitable} is made accessible through the following theorem:

\begin{theorem}[Butterfly theorem]\label{Butterfly}
	Let $G_2$ be a finite group with normal subgroup $N$. Let $(G_1,N, \theta)$ and $(H_1,M, \theta')$ be two character triples with $(G_1,N, \theta)\geq_b(H_1,M, \theta')$. Assume that via the canonical morphism $\varepsilon_i:G_i \to \mathrm{Aut}(N)$, $i=1,2$, we have $\varepsilon_1(G_1)=\varepsilon_2(G_2)$. Then for $H_2:= \varepsilon_2^{-1} \varepsilon_1(H_1)$ we have
	$$(G_2,N,\theta) \geq_b (H_2,M,\varphi).$$
\end{theorem}

\begin{proof}
	See \cite[Theorem 2.16]{LocalRep} and \cite[Theorem 4.6]{LocalRep}.
\end{proof}

The inductive Alperin--McKay condition was introduced by Späth in \cite[Definition 7.12]{IAM}. We will in the following use its reformulation in the language of character triples, see \cite[Definition 4.12]{LocalRep}.

\begin{definition}
	Let $S$ be a non-abelian simple group with universal covering group $G$ and $b$ a block of $G$ with non-central defect group $D$. If there exists a strong iAM-bijection $\Psi: \Irr_0(G , b) \to \Irr_0(M , B)$ for the block $b$ with respect to a subgroup $M$ then we say that the \textit{block $b$ is AM-good for $\ell$}.
	%\begin{enumerate}[label=(\alph*)]
	%%\item some $\Gamma$-stable subgroup $M$ with $\mathrm{N}_G(D) \leq M \lneq G$
	%\item a $\Gamma$-equivariant bijection $\Psi: \Irr_0(b) \to \Irr_0(B_D)$ where $B_D$ is the Brauer correspondent of $b$. 
	%\item $\Psi( \Irr_0(b  \mid \nu) ) \subseteq \Irr_0(B_D \mid \nu )$ for every $\nu \in \Irr(\mathrm{Z}(G))$ and
	%$$( G/ Z \rtimes \Gamma_\chi, G/Z, \overline{\chi}) \geq_{b} ( \mathrm{N}_{G}(D)/Z \rtimes \Gamma_\chi, \mathrm{N}_G(D)/Z, \overline{\Psi(\chi)} ),$$
	%for every $ \chi \in \Irr_0(b)$ and $Z= \mathrm{ker}(\chi_{\mathrm{Z}(G)})$, where $\overline{\chi}$ and $\overline{\Psi(\chi)}$ lift to $\chi$ and $\Psi(\chi)$, respectively.
	%\end{enumerate}
	
\end{definition}

\subsection{Normal subgroups and character triple bijections}

In this section we recall two general statements from the theory of character triples which we will use in the next sections.
%The notation will therefore be unrelated to the notation of the previous sections.

The following proposition is a variant of \cite[Proposition 4.7(b)]{JEMS}:

\begin{proposition}\label{JEMSproposition}
	Let $X$ be a finite group. Suppose that $N \lhd X$ and $H \leq X$ such that $X=NH$. Let $e$ be a block of $N$ with defect group $D_0$. Assume that $M:=N \cap H$ satisfies $H=M \mathrm{N}_X(D_0)$ and let $E \in \mathrm{Bl}(M \mid D_0)$ such that $E^N=e$. Suppose there exists an $\mathrm{N}_X(D_0,E)$-equivariant bijection
	$$\varphi: \Irr_0(N , e) \to \Irr_0( M , E)$$
	such that $(X_\theta,N,\theta) \geq_{b} (H_\theta,M, \varphi(\theta))$.
	Furthermore let $J \lhd X$ such that $N \leq J$. Let $c$ be a block of $J$ covering $e$ and $D$ a defect group of $c$ satisfying $D \cap N=D_0$ and let $C \in \mathrm{Bl}(J \cap H \mid D)$ with $C^J=c$. Then there exists an $\mathrm{N}_H(D,C)$-equivariant bijection
	$$\pi_J: \Irr_0(J , c) \to \Irr_0( J \cap H , C)$$
	such that $(X_\tau,J, \tau) \geq_b (H_\tau, J \cap H, \pi_J(\tau))$ for all $\tau \in  \Irr_0(J , c)$.
	%MAYBE IN SPÃTH VERSION
	%Let $K \lhd X$. Let $c_0$ be a block of $K$ ... defect group $D_0$ and $C_0$ its first Brauer correspondent. Assume that there exists an $\mathrm{N}_K(D_0)$-equivariant bijection
	%$$\varphi: \Irr_0(K , c_0) \to \Irr_0( \mathrm{N}_K(D_0) , (C_0)_{D_0})$$
	%such that $(X_\theta,K,\theta) \geq_{b} (H_\theta,\mathrm{N}_K(D_0), \varphi(\theta))$. Moreover, assume that $J \lhd X$ with $K \leq J$.
	%%Let $Q \leq J$ be an $\ell$-subgroup with $Q \cap K =D_0$.
	%Let $c$ be a block of $J$ covering $c_0$ with defect group $D$ satisfying $D \cap K=D_0$ and $C_{D_0}$ the Harris--KnÃ¶rr correspondent of $c$ in $\mathrm{N}_J(D_0)$. Then there exists an $\mathrm{N}_J(D_0)$-equivariant bijection
	%%$$\pi_J: \Irr(J \mid \Irr_0(K \mid e^K)) \to \Irr(J \cap H \mid \Irr_0(M \mid e))$$
	%$$\pi_J: \Irr_0(J , c) \to \Irr_0( \mathrm{N}_J(D_0) , C_{D_0})$$
	%%I changed N_J(D_0) to N_J(D). Changed it back
	%such that $(X_\tau,J, \tau) \geq_b (H_\tau, \mathrm{N}_J(D_0), \pi_J(\tau))$.
\end{proposition}

\begin{proof}
	This is proved as in \cite[Proposition 4.7(b)]{JEMS}. We note that the assumptions in in \cite[Proposition 4.7(b)]{JEMS} are stronger. However, one can with our weaker assumption and with the same proof show that the conclusion of the proposition holds.
	%Let $\tau \in \Irr_0( J , e^J)$. Then there exists a character $\theta \in \Irr_0(K , e_0^K)$ and a character $\tau_0 \in \Irr(J_\theta , \theta)$ such that $\tau=\mathrm{Ind}_{J_\theta}^J(\tau_0) \in \Irr_0(J , e^J)$.
	%We obtain an $H_\theta$-equivariant map
	%$$\sigma_{J_\theta}^{(\theta)}: \Irr(J_\theta \mid \theta) \to \Irr(J_\theta \cap H_\theta \mid \varphi(\theta))$$
	%such that $((X_\theta)_{\tau_0},J_\theta,\tau_0)) \geq_b ((H_\theta)_{\tau_0},J_\theta \cap H, \sigma_{J_\theta}^{(\theta)}(\tau_0))$. Denote $\tau_0':= \sigma_{J_\theta}^{(\theta)}(\tau_0)$ and $\tau':= \mathrm{Ind}_{ \mathrm{N}_J(D_0)_\theta}^{\mathrm{N}_J(D_0)}(\tau_0')$.
	%
	%%By \cite[Proposition 2.5(a)]{JEMS} and \cite[Proposition 2.5(f)]{JEMS} there exists
	%By \cite[Theorem 3.14]{JEMS} the induced characters satisfy
	%$$(X_\tau,J,\tau) \geq_b (H_\tau,J \cap H, \tau')$$
	%We can thus define a map
	%$$\pi_J: \Irr(J , c) \to \Irr(J \cap H , (C_0)_{D_0})$$
	%by defining $\pi_J(\tau):= \tau'$.
	%The fact that $\pi_J$ is a bijection is 
\end{proof}

%Using Lemma \ref{Coro24} we now reformulate the conditions of Theorem \ref{reductionquasi} in terms of character triples.

Let $G$ be a finite group. As in \cite{JEMS} the set $\Irr_0(G \mid D)$ denotes the set of height zero characters in $\ell$-blocks of $G$ with defect group $D$. We use the following statement, which is a consequence of the Dade--Nagao--Glauberman correspondence.

%
% Recall that if $N \lhd G$ such that $G/N$ is an $\ell$-group then every $\ell$-block of $N$ is covered by a unique block of $G$, see \cite[Corollary 9.6]{NavarroBook}.

\begin{lemma}\label{DGN}
	Let $X$ be a finite group, $M \lhd X$ and $N \lhd M$ such that $M/N$ is an $\ell$-group. Let $D_0 \lhd N$ be an $\ell$-subgroup with $D_0 \leq \mathrm{Z}(M)$.
	
%	Suppose that $c_0$ is an $M$-invariant block of $N$. Let $e$ be the unique block of $M$ covering $c_0$.
	
	 Then for every $\ell$-subgroup $D$ of $M$ there exists an $\mathrm{N}_X(D)$-equivariant bijection
	$$\Pi_D: \Irr_0(M \mid D) \to \Irr_0(\mathrm{N}_M(D) \mid D)$$
	with $(X_\tau, M, \tau) \geq_b (\mathrm{N}_X(D)_\tau, \mathrm{N}_M(D), \Pi_D( \tau ))$ for every $\tau \in \Irr_0(M  \mid D)$.
\end{lemma}

\begin{proof}
	This is \cite[Corollary 5.14]{JEMS}.
\end{proof}

\subsection{A criterion for block isomorphic character triples}

In this section we establish a slightly more general version of \cite[Lemma 3.2]{CS14}. This will be required to obtain Lemma \ref{12} below in the case where $\G^F$ is of type $D_4$. Thus, the reader who is not interested in the specifics of this case may skip this section entirely.

Let $A$ be a group acting on a finite group $H$. Then we denote by $\mathrm{Lin}_A(H)$ the subset of $A$-invariant linear characters in $\Irr(H)$. Moreover, if $N$ is a normal subgroup of $A$ and $b$ is a block of $N$, then we denote by $A[b]$ \textit{the ramification group of the block} $b$ in $A$, which was introduced by Dade, see \cite[Definition 3.1]{CS14}. 

\begin{lemma} \label{newred}
	Let $A$ be a finite group. Suppose that $N\unlhd A$ and $N\leq J\unlhd A$ such that $J/N$ is abelian and $A/N$ is solvable. Assume that $\mathrm{Lin}_A(H)=\{1_H\}$ for every subgroup $H$ of the quotient group $([A,A] J) / J$. Let $b$ be a block of $N$ and suppose that Hall $\ell'$-subgroup $I/N$ of $A[b]/N$ satisfies $A=\mathrm{N}_A(I) J$. We let $\chi, \phi \in \Irr(N,b)$ and $\tilde{\chi}\in \Irr(A)$ and $\tilde{\phi}\in\Irr(A[b])$ be extensions of $\chi$ and $\phi$ respectively with $\bl (\Res^A_{J_1}(\tilde{\chi})\bigr) =\bl(\Res^{A[b]}_{J_1}(\tilde{\phi}))$ for every $J_1$ with $N\leq J_1 \leq J[b]$. Assume moreover that $\mathrm{bl}(\Res_{I}^{A[b]}(\tilde \phi))$ is $\mathrm{N}_A(I)$-stable. Then there exists an extension $\tilde{\chi}_1 \in \Irr(A)$ of $\Res^A_{J}(\tilde{\chi})$ with 
	$$\bl(\Res^A_{J_2}(\tilde{\chi}_1)\bigr)=\bl\bigl(\Res^{A[b]}_{J_2}(\tilde{\phi}))$$
	for every $J_2$ with $N\leq J_2\leq A[b]$.
\end{lemma}
\begin{proof} 
	We copy the first part of the proof of \cite[Lemma 3.2]{CS14}. Since $A/N$ is solvable, there exists some group $I$ with $N\leq I\leq A[b]$ such that $I/N$ is a Hall $\ell'$-subgroup of $A[b]/N$ and $(I\cap J )/N$ is a Hall $\ell'$-subgroup of $J[b]/N$, see \cite[Theorem 18.5]{Aschbacher}. 
	According to \cite[Theorem C(b)(1)]{KS} there exists an extension $\tilde{\chi}_2\in\Irr(I)$ of $\chi$ to $I$ with $\bl(\tilde{\chi}_2)= \bl(\Res^{A[b]}_{I}(\tilde{\phi}))$.
	This extension also satisfies $\bl(\Res^{I}_{I\cap J}(\tilde{\chi}_2))=\bl(\Res^{A[b]}_{I\cap J}(\tilde{\phi}))$ according to \cite[Lemma 2.4]{KS} and \cite[Lemma 2.5]{KS}.
	By \cite[Lemma 3.7]{KS} there is a unique character in $\Irr(I\cap J\mid \chi)$ with this property, hence $\Res^{A}_{I\cap J}(\tilde{\chi})=\Res^{A[b]}_{I\cap J}(\tilde{\chi}_2)$ by the assumptions on $\tilde{\chi}$. By \cite[Lemma 5.8(a)]{Equivariantcharacter} there exists an extension $\eta\in\Irr(IJ)$ of $\Res^{A}_{J}(\tilde{\chi})$ such that $\Res^{IJ}_I (\eta )=\tilde{\chi}_2$.
	
	Since $\eta$ and $\mathrm{Res}_{IJ}^{A}( \tilde{\chi})$ are both extensions of the character $\Res^{A}_{J}(\tilde{\chi})$ there exists by Gallagher's theorem a unique linear character $\mu \in \Irr(IJ/ J)$ such that $\eta = \mathrm{Res}_{IJ}^{A}( \tilde{\chi})  \mu$. We claim that $\mu$ is $A$-stable.

	We obtain $\tilde{\chi}_2=\mathrm{Res}_{I}^{IJ}(\eta) = \mathrm{Res}_{I}^{A}( \tilde{\chi})  \mathrm{Res}_{I}^{IJ}(\mu)$. Recall that $\bl(\tilde{\chi}_2)= \bl(\Res^{A[b]}_{I}(\tilde{\phi}))$ is $\mathrm{N}_A(I)$-stable since $\bl(\Res^{A[b]}_{I}(\tilde{\phi}))$ is assumed to be $\mathrm{N}_A(I)$-stable. Since $\tilde{\chi}_2$ is the unique extension of $\chi$ lying in this block by \cite[Theorem 4.1]{Murai} it must be $\mathrm{N}_A(I)$-stable as well. From this it follows that $\mathrm{Res}_{I}^{IJ}(\mu)$ and therefore $\mu$ is $\mathrm{N}_A(I)$-stable. Since $IJ/J \cong I/I \cap J$ it follows that $\mu$ is $J \mathrm{N}_A(I)$-stable. Since $A=J \mathrm{N}_A(I)$ by assumption it follows that $\mu$ is $A$-stable as claimed.
		
	 Now, $([A,A] J \cap IJ) / J \leq ([A,A] J) / J$. By assumption every $A$-invariant linear character of a subgroup of $([A,A] J) /J$ is trivial. Hence we obtain $\mathrm{Res}^{IJ}_{[A,A] J \cap IJ}(\mu)=1_{[A,A] J \cap IJ}$. In other words, $[A,A] J \cap IJ$ is contained in the kernel of the linear character $\mu$. Hence we can consider $\mu$ as a character of $IJ/([A,A]J \cap IJ)$. Since $IJ/([A,A]J \cap IJ) \hookrightarrow A /[A,A]$ there exists a linear character $\tilde{\mu} \in \Irr(A)$ extending $\mu$. We define $\tilde{\chi}_1:= \tilde{\mu} \tilde{\chi}$ and obtain
	$$\mathrm{bl}(\Res_{I}^A (\tilde{\chi}_1 ))=\mathrm{bl}(\mathrm{Res}_{I}^{A}( \tilde{\chi})  \Res_I^{IJ}(\mu))=\mathrm{bl}(\tilde{\chi}_2)=\mathrm{bl}(\Res_{I}^{A[b]}(\tilde \phi)) .$$
	%Since $\Res^{A}_{J}(\tilde{\chi})$ extends to $A$ and $A/J$ is abelian, every character of $\Irr(A\mid\Res^{A}_{J}(\tilde{\chi}))$ is an extension of $\Res^{A}_{J}(\tilde{\chi})$.
	%Since $\Irr(A\mid \eta)\subseteq\Irr(A\mid \Res^{A}_{J}(\tilde{\chi}))$ there exists an extension $\w\eta$ of $\eta$ to $A$.
	After having constructed the extension $\tilde{\chi}_1$ the same arguments from \cite[Lemma 3.2]{CS14} show the result. For the convenience of the reader we will recall the arguments here.
	
	According to \cite[Lemma 2.4]{KS} (which also holds for ordinary characters instead of Brauer characters) the character $\tilde{\chi}_1$ satisfies
	\[ \bl (\Res_{\langle N,x \rangle }^A( \tilde{\chi}_1) )=\bl\bigl(\Res^{A[b]}_{\langle N,x \rangle }( \tilde{\phi}) ) \text{ for every }x \in I \text{ of $\ell'$-order}.\]
	Since $I/N$ is a Hall $\ell'$-subgroup of $A[b]/N$ it follows that every element $x\in A[b]$ of $\ell'$-order is conjugate to some element in $I$. Consequently, the above equality holds for every element $x\in A[b]$ of $\ell'$-order. By \cite[Lemma 2.5]{KS} this implies
	\begin{align*}
	\bl\bigl(\Res^A_{J_2} (\tilde{\chi}_1))&=\bl(\Res^{A[b]}_{J_2}( \tilde{ \phi} )) \text{ for every group $J_2$ with }N\leq J_2\leq A[b].
	\qedhere
	\end{align*}
\end{proof}
%
%\begin{remark}
Note that if $A/J$ is abelian the assumption that $\Irr_A(H)=\{1_H\}$ for every subgroup $H$ of the quotient group $([A,A] J) / J$ is trivially satisfied. Hence, we obtain as a special case the original statement of \cite[Lemma 3.2]{CS14}.
%\end{remark}

\begin{remark}\label{newred3}
In this remark we consider the situation when $A/N =S_a \times C_m$ for some $m$ and $a \in \{ 3,4\}$. Moreover, let $J=N$ if $a=3$ and $J=V_4$ if $a=4$. Firstly, observe that $([A,A]J)/J=[A/J, A/J] \cong C_3$. Since no non-trivial character of $C_3$ is fixed by $S_3$ we thus have $\mathrm{Lin}_A(H)=\{1_H\}$ for every subgroup $H$ of $[A,A] J / J$. A calculation inside $S_4 \times C_m$ shows that the Hall $\ell'$-subgroup $I$ of $A[b]/N$ satisfies $\mathrm{N}_A(I)J=A$ whenever $\ell \neq 3$. Thus, the group theoretic requirements of Lemma \ref{newred} are satisfied unless possibly if $\ell=3$.
  
We check that the conclusion in Lemma \ref{newred} nevertheless still holds for $\ell=3$ in this situation. For this observe that we only needed the assumption on $I$ to construct the character $\tilde{\chi}_1$. Recall that $\eta\in\Irr(IJ)$ is an extension of both $\Res^{A}_{J}(\tilde{\chi})$ and $\tilde{\chi}_2$. Let $P/J$ be the normal Sylow $\ell$-subgroup of $A/J$. By \cite[Lemma 5.8(a)]{Equivariantcharacter} there exists an extension $\eta' \in \Irr(IJP)$ extending both $\eta$ and $\Res^{A}_{JP}(\tilde{\chi})$. By the proof of Lemma \ref{newred} the character $\eta$ is $\mathrm{N}_{A}(I)J$-stable and $ \Res^{A}_{JP}(\tilde{\chi})$ is clearly $A$-stable. Consequently, $\eta'$ is $P\mathrm{N}_{A}(I)J$-invariant. By observing the group structure of $A/J \cong S_3 \times C_m$ we can now deduce that $A=P \mathrm{N}_A(I)J$ and thus $\eta'$ is $A$-stable. Note that $A/JP$ is abelian and $\eta'$ is an extension of $\Res^{A}_{JP}(\tilde{\chi})$. Thus, $\eta'$ extends to a character $\tilde{\eta}$ of $A$. We deduce that
	$$\mathrm{bl}(\Res_{I}^A(\tilde{\eta}))=\mathrm{bl}(\tilde{\chi}_2)=\mathrm{bl}(\Res_{I}^{A[b]}(\tilde \phi)) $$
and thus $\tilde \eta$ is a suitable extension.
\end{remark}

\section{The first reduction theorem}

\subsection{A condition on the stabilizer and the inductive conditions}

%In this section we introduce one of the most important results which is used in practice to verify the inductive Alperin--McKay condition for simple groups of Lie type.

From now on $\G$ denotes a simple algebraic group of simply connected type with Frobenius $F: \G \to \G$ and such that the finite group $\G^F/ \mathrm{Z}( \G^F)$ is non-abelian and simple with universal covering group $\G^F$. We let $\iota: \G \hookrightarrow \tilde{\G}$ be a regular embedding as in \cite[4.1]{Jordan}. For every closed $F$-stable subgroup $\mathbf{H}$ of $\tilde{\G}$ we write $H:= \mathbf{H}^F$. In the following we denote by $\mathcal{B}:= \langle \mathcal{G}, \phi_0 \rangle \subseteq \mathrm{Aut}(\tilde{\G}^F)$ the subgroup generated by the group $\mathcal{G} \subseteq \mathrm{Aut}(\tilde{\G}^F)$ of graph automorphisms and the field automorphism $\phi_0: \tilde{\G}^F \to \tilde{\G}^F$ from \cite[4.1]{Jordan}. By the description of automorphisms of simple simply connected algebraic groups in \cite[Theorem 2.5.1]{GLS} it follows that $\mathrm{C}_{\tilde{G}\mathcal{B}}(G)= \mathrm{Z}(\tilde{G})$.

%We need a version of Lemma \ref{12} for groups of type $D_4$. The proof of $"\geq_c"$ works in the same way in type $D_4$. However, the proof of $"\geq_b"$ relies on an application of \cite[Lemma 3.2]{CS14}. In particular, it requires the group $\mathcal{B}_\chi$ to be abelian, see proof of Lemma \ref{12}. To remedy this situation we provide a generalization of \cite[Lemma 3.2]{CS14} tailored to our situation.

Let $\G^\ast$ be a connected reductive group in duality with $\G$ and $F^\ast: \G^\ast \to \G^\ast$ a Frobenius endomorphism dual to $F$.
We fix a semisimple element $s \in (\G^\ast)^{F^\ast}$ of $\ell'$-order and a block $b$ of $\mathcal{O} \G^F e_s^{\G^F}$ with defect group $D$. We let $Q$ be a characteristic subgroup of $D$. In the following we abbreviate $M:= \mathrm{N}_G(Q)$ and $\tilde{M}:= \mathrm{N}_{\tilde{G}}(Q)$. Moreover, $B_Q \in \mathrm{Bl}(M \mid Q)$ denotes the unique block with $(B_Q)^M=b$.

The following theorem is essentially due to Cabanes--Späth \cite{CS14}. In previous work this theorem has turned out to be useful in the verification of the inductive Alperin--McKay condition for simple groups of Lie type.
%Recall that $\mathcal{A} \subseteq \mathrm{Aut}(\tilde{\G}^F)$ is a group of automorphisms stabilizing $\Levi$ and $e_s^{\Levi^F}$.
%
%\begin{lemma}
%We have $\mathrm{C}_{\tilde{G} \mathcal{A}}(G)=\mathrm{Z}(\tilde{G}) \mathcal{B}$ where $\mathcal{B}=\langle l \rangle \subseteq \mathcal{A}$, for some $l \in L$.  
%\end{lemma}
%
%\begin{proof}
%
%\end{proof}

\begin{theorem}\label{12}
	Let $ \chi \in \Irr(G,b)$ and $\chi' \in \Irr(M,B_Q)$ such that the following holds:
	\begin{enumerate}[label=(\roman*)]
		\item We have $(\tilde{G} \mathcal{B})_\chi = \tilde{G}_\chi \mathcal{B}_\chi$ and 
		$\chi$ extends to $(G \mathcal{B})_\chi$.
		\item
		We have $( \tilde{M}  \mathrm{N}_{G \mathcal{B}}(Q) )_{\chi'}= \tilde{M}_{\chi'} \mathrm{N}_{G \mathcal{B}}(Q)_{\chi'}$ and $\chi'$ extends to $ \mathrm{N}_{G \mathcal{B}}(Q)_{\chi'}$.
		\item $(\tilde{G} \mathcal{B})_\chi = G ( \tilde{M}  \mathrm{N}_{G \mathcal{B}}(Q) )_{\chi'}$.
		\item There exists $\tilde{\chi} \in \Irr(\tilde{G}_\chi \mid \chi)$ and $\tilde{\chi}' \in \Irr(\tilde{M}_{\chi'} \mid \chi')$ such that the following holds:
		\begin{itemize}

			\item  For all $m \in \mathrm{N}_{G \mathcal{B}}(Q)_{\chi'}$ there exists $\nu \in \Irr(\tilde{G}_{\chi} /G)$	with ${\tilde{\chi}}^m= \nu \tilde{\chi}$ and $\tilde{\chi}'^m=\mathrm{Res}^{\tilde{G}_{\chi}}_{\tilde{M}_{\chi'}}(\nu) \tilde{\chi}'$.
			\item The characters $\tilde{\chi}$ and $\tilde{\chi}'$ cover the same underlying central character of $\mathrm{Z}(\tilde{G})$.
		\end{itemize}
		\item For all $G \leq J \leq \tilde{G}_{\chi}$ we have $\mathrm{bl}(\mathrm{Res}^{{\tilde{G}_\chi}}_{{J}}({\tilde{\chi}}))= \mathrm{bl}(\mathrm{Res}^{\mathrm{N}_{\tilde{G}}(Q)_{\chi'}}_{\mathrm{N}_{{J}}(Q)}({\tilde{\chi}'}))^{{J}}$.
	\end{enumerate}
	Let $\mathrm{Z}:=  \mathrm{Ker}(\chi) \cap \mathrm{Z}(G) $.
	Then
	$$(( \tilde{G} \mathcal{B})_\chi /Z, G/Z, \overline{\chi}) \geq_b ((\tilde{M} \mathrm{N}_{G \mathcal{B}}(Q))_{\chi'} / Z,M/Z, \overline{\chi'}),$$
	where $\overline{\chi}$ and $\overline{\chi'}$ are the characters which inflate to $\chi$, respectively $\chi'$.
\end{theorem}

\begin{proof}
	If assumptions (i)-(iv) hold then we have
	$$(( \tilde{G} \mathcal{B})_\chi /Z, G/Z, \overline{\chi}) \geq_c ((\tilde{M} \mathrm{N}_{G \mathcal{B}}(Q))_{\chi'} / Z,M/Z, \overline{\chi'}),$$
	by \cite[Lemma 2.7]{S12}.
	It therefore remains to show that the additional property in Definition \ref{blockisomorphism}(ii) holds in order to show that the relation $\geq_b$ holds as well. For this, we go through the proof of \cite[Proposition 4.2]{CS14} which is still applicable under our assumptions since we can replace \cite[Equation (4.2)]{CS14} in the proof of \cite[Proposition 4.2]{CS14} by our assumption (v).
	
	Now we want to apply the proof of \cite[Theorem 4.1]{CS14}. In the notation of  \cite[Proposition 4.2]{CS14}, Cabanes--Späth construct a group $A$ together with a central extension $\varepsilon: A \to \mathrm{Aut}(G)_\chi$ of $\mathrm{Aut}(G)_\chi$. Denote by $\mathrm{Aut}_{\tilde{G}_\chi}(G)$ the subgroup of automorphisms of $\mathrm{Aut}(G)$ induced by the conjugation action of $\tilde{G}_\chi$. We set $J:= \varepsilon^{-1}(\mathrm{Aut}_{\tilde{G}_\chi}(G))$ amd $N:=\varepsilon^{-1}(\mathrm{Inn}(G))$. By construction, $A/J \cong \mathcal{B}_\chi$, which is abelian unless possibly if $G$ is of type $D_4$.
	
	If $A/J$ is abelian then the group-theoretic assumptions of \cite[Lemma 3.2]{CS14} are satisfied. Thus, we can apply the proof of \cite[Theorem 4.1]{CS14} without change and we deduce that the characters $\chi$ and $\chi'$ satisfy the conditions in \cite[Definition 2.1(c)]{CS14}.

	If $A/J$ is non-abelian then as argued above $G$ is of type $D_4$ and it follows that $A/N \cong S_a \times C_m$ for some integer $m$ and $a\in \{3,4\}$ and $J=N$ if $a=3$ or $J=V_4$ if $a=4$, i.e. we are in the situation of Remark \ref{newred3}. We go through the proof of \cite[Theorem 4.1]{CS14} again. Consider the character $\tilde \phi \in \Irr(A[b])$ constructed in the proof of \cite[Theorem 4.1]{CS14}. By its definition given there one can see that $\mathrm{bl}(\Res_I^{A[b]}\tilde \phi)$ is indeed $A$-stable. Thus, instead of applying \cite[Lemma 3.2]{CS14} in the proof of  \cite[Theorem 4.1]{CS14} we use Lemma \ref{newred} and Remark \ref{newred3}. This shows that also in this case the characters $\chi$ and $\chi'$ satisfy the conditions in \cite[Definition 2.1(c)]{CS14}. 
	
	However, since the characters $\chi$ and $\chi'$ satisfy the conditions in \cite[Definition 2.1(c)]{CS14} it follows by Theorem \ref{Butterfly} that the additional property in Definition \ref{blockisomorphism}(ii) holds. This finishes the proof.
	%Hence, if $If $A/J$ is abelian\mathcal{B}_\chi$ is non-abelian it follows that $\mathcal{B}_\chi \cong S_3 \times C_m$ for some integer $m$. 
	%
	%Thus, we obtain $\w\chi\in\Irr(A)$ an extension of the character $\o\chi\in\Irr(\o G)$ that lifts to $\chi$. Moreover, there exists an extension $\w\psi\in\Irr\bigl(\o M \NNN_A(\o D)\bigr)$ of the character $\o\psi \in \Irr(\o M)$ that lifts to $\psi$ where $\o M:=M/Z$ and $\o D:=DZ/Z$. These characters satisfy
	% \[\bl\bigl(\Res^{A}_{J_2}(\w\chi)\bigr)=\bl\bigl(\Res^{\o M \NNN_A(\o D)}_{\o M \NNN_{J_2}(\o D)}(\w\psi)\bigr)^{J_2}\text{ for every $J_2$ with }\o G\leq J_2 \leq J.\]
\end{proof}

%\begin{proposition}\label{Blockinduction}
%
%Suppose that we are in the situation of Theorem \ref{12}. Let $\mathrm{Z}:= \mathrm{Z}(G) \cap \mathrm{Ker}(\chi)$. Assume that
%%\begin{enumerate}[label=(\alph*)]
%%\item $\mathrm{bl}(\chi)= \mathrm{bl}(\chi')^G$ and
%$$\mathrm{bl}(\mathrm{Res}^{{\tilde{G}}}_{{J}}({\tilde{\chi}}))= \mathrm{bl}(\mathrm{Res}^{\mathrm{N}_{{\tilde{G}}}(Q)}_{\mathrm{N}_{{J}}(Q)}({\tilde{\chi}'}))^{{J}}$ for all $G \leq J \leq \tilde{G}_{\chi}$. Then we have
%%$$(( \tilde{G} E)_\chi , G, \chi) \geq_b ((\tilde{M} \mathrm{N}_{G E}(Q))_{\chi'},M, \chi').$$
%$$(( \tilde{G} E)_\chi /Z, G/Z, \overline{\chi}) \geq_b ((\tilde{M} \mathrm{N}_{G E}(Q))_{\chi'} / Z,M, \overline{\chi'}),$$
%where $\overline{\chi}$ and $\overline{\chi'}$ are the characters which inflate to $\chi$, respectively $\chi'$.
%\end{proposition}
%
%\begin{proof}
%
%\end{proof}

%\begin{corollary}\label{central subgroup}
%
%%Denote by $\bar{\,}: \tilde{G}_{\chi} \to \tilde{G}_{\chi}/ Z$ the canonical map.
%Under the assumptions of Proposition \ref{Blockinduction}
%we have 
%\end{corollary}
%
%\begin{proof}
%Since $\mathrm{Z}(G)$ is a central subgroup of $\tilde{G}$ it follows that $\mathrm{Z}$ is a central subgroup of $J$. By \cite[Lemma 10.19]{NavarroNewBook} we have
%$$(( \tilde{G} E)_\chi /Z, G/Z, \overline{\chi}) \geq_c ((\tilde{M} \mathrm{N}_{G E}(D))_{\chi'} / Z,M, \overline{\chi'}).$$
%The statement about $" \geq_b"$ now follows from \cite[Proposition 2.4(b)]{JEMS} and \cite[Proposition 2.4(c)]{JEMS}.
%\end{proof}

We will check condition (v) in Theorem \ref{12} using the following:

\begin{lemma}\label{Blockinduction2}
	Let $\chi \in \Irr(G,b)$ and $ \chi' \in \Irr(\mathrm{N}_G(Q),B_Q)$ such that $\tilde{G}_\chi=G \tilde{M}_{\chi'}$. Let $\tilde{\chi} \in \Irr(\tilde{G}_\chi \mid \chi)$ be an extension of $\chi$ and $\tilde{\chi}' \in \Irr(\tilde{M}_{\chi'} \mid \chi')$ be an extension of $\chi'$ such that
	% $\mathrm{bl}(\chi)= \mathrm{bl}(\chi')^G$.
	%% and 
	$\mathrm{bl}(\tilde{\chi})= \mathrm{bl}(\tilde{\chi'})^{\tilde{G}_\chi}$.
	Then we have
	$$\mathrm{bl}(\mathrm{Res}^{\tilde{G}_\chi}_{J}(\tilde{\chi}))= \mathrm{bl}(\mathrm{Res}^{\mathrm{N}_{\tilde{G}}(Q)_{\chi'}}_{\mathrm{N}_J(Q)}(\tilde{\chi}'))^{J}$$ for all $G \leq J \leq \tilde{G}_{\chi}$.
	%Let $N$ be a normal subgroup of a finite group $G$ such that $G/N$ is abelian and $\chi \in \Irr(N)$ an irreducible character which extends to a character $\tilde{\chi}$ of $G$.
	%Then for every $J$ between $N$ and $G$ the following holds:
	%\begin{enumerate}
	%\item 
	%The block $\mathrm{bl}(\mathrm{Res}^G_J(\tilde{\chi}))$ is the unique block which covers $\mathrm{bl}(\chi)$ and is covered by $\mathrm{bl}(\tilde{\chi})$.
	%\item If $e \in \mathrm{Bl}(G \mid b)$ is the Harris-Kn\"orr correspondent of $\mathrm{bl}(\tilde{\chi})$ then the unique block $e_J \in \mathrm{Bl}(J \mid B_Q)$ covered by $e$ is the Harris--KnÃ¶rr correspondent of $\mathrm{bl}(\mathrm{Res}^G_J(\tilde{\chi}))$.
	%\end{enumerate}
\end{lemma}

\begin{proof}
	Since $\tilde{G}/G$ is abelian, it follows that $J$ is a normal subgroup of $\tilde{G}_\chi$. Hence, $\mathrm{bl}(\mathrm{Res}^{\tilde{G}_\chi}_J(\tilde{\chi}))$ is the unique block which is covered by $\mathrm{bl}(\tilde{\chi})$.
	%Moreover, $\mathrm{bl}(\mathrm{Res}^{\tilde{G}_\chi}_J(\tilde{\chi}))$ covers the block $\mathrm{bl}(\chi)$.
	On the other hand, $\mathrm{bl}(\mathrm{Res}^{\mathrm{N}_{\tilde{G}}(Q)_{\chi'}}_{\mathrm{N}_J(Q)}(\tilde{\chi}'))$ is the unique block of $\mathrm{N}_J(Q)$ which is covered by $\mathrm{bl}(\tilde{\chi}')$.
	%
	% and which covers $\mathrm{bl}(\chi')$.
	%Since  are Brauer correspondents and
	Moreover, we have $\mathrm{bl}(\chi')^G=\mathrm{bl}(\chi)$ and $\mathrm{bl}(\tilde{\chi}')^{\tilde{G}_\chi}=\mathrm{bl}(\tilde{\chi})$. We deduce that
	%$\mathrm{bl}(\mathrm{Res}^{\mathrm{N}_{\tilde{G}}(Q)_{\chi'}}_{\mathrm{N}_J(Q)}(\tilde{\chi}'))^J=\mathrm{bl}(\mathrm{Res}_{\mathrm{N}_J(\tilde{G}_\chi)}^{\mathrm{N}_{\tilde{G}_\chi}}(\tilde{\chi}'))$ are Harris--Kn\"orr correspondents. Thus, we have 
	$\mathrm{bl}(\mathrm{Res}^{\tilde{G}_\chi}_J(\tilde{\chi}))= \mathrm{bl}(\mathrm{Res}^{\mathrm{N}_{\tilde{G}}(Q)_{\chi'}}_{\mathrm{N}_J(Q)}(\tilde{\chi}'))^{J}$.
\end{proof}

\subsection{Extension of characters}

Suppose that $\G^F$ is not of untwisted type $D_4$. Let $\Levi^\ast$ be the minimal Levi subgroup of $\G^\ast$ containing $\mathrm{C}(s):=\mathrm{C}_{(\G^\ast)^{F^\ast}}(s)\mathrm{C}^\circ_{\G^\ast}(s)$.

 Let $\sigma: \tilde{\G} \to \tilde{\G}$ and $F_0: \tilde{\G} \to \tilde{\G}$ be the automorphisms constructed in \cite[Proposition 4.9]{Jordan} and denote by $\mathcal{A} \subseteq \mathrm{Aut}(\tilde{\G}^F)$ the subgroup generated by these automorphisms.
 Then \cite[Proposition 4.9]{Jordan} shows that there exists a Levi subgroup $\Levi$ of $\G$ in duality with the Levi subgroup $\Levi^\ast$ such that $\mathcal{A}$ stabilizes $\Levi$ and $e_s^{\Levi^F}$.

%Recall that there exists a canonical bijection $\mathrm{Z}(\tilde{G}) \to \Irr(\tilde{G}/G), \, z \mapsto \hat{z}$, see for instance \cite[Equation 8.19]{MarcBook}.

The next lemma shows that extendibility to $G \mathcal{A}$ can be compared with extendibility to $G \mathcal{B}$. In the following, we denote by $\mathrm{ad}(x): G \to G$ the inner automorphism of $\tilde{G}$ given by conjugation with $x \in G$.

\begin{lemma}\label{AE}
	Let $\chi \in \Irr(\G^F , e_s^{\G^F})$. Then the character $\chi$ extends to $G \mathcal{A}_\chi$ if and only if it extends to $G \mathcal{B}_\chi$.
\end{lemma}

\begin{proof}
	By \cite[Proposition 4.9]{Jordan} the image of $\tilde{G} \rtimes \mathcal{A}$ in $\mathrm{Out}(G)$ is the stabilizer of $e_s^{\G^F}$ in $\mathrm{Out}(G)$. From this it follows that $\mathcal{A}_\chi$ and $\mathcal{B}_\chi$ generate the same group in $\mathrm{Out}(G)$. Thus, if $\mathcal{A}_\chi$ is cyclic then so is $\mathcal{B}_\chi$ and $\chi$ extends in both cases. Now assume that $\mathcal{A}_\chi$ is non-cyclic. Then by the proof of \cite[Proposition 4.9]{Jordan} there exists some $x\in \G^{F_0}$ such that $\mathcal{A}=\langle \mathrm{ad}(x) \gamma, F_0 \rangle$, where $F_0 \in \langle \phi_0 \rangle$ and $\gamma \in \mathcal{B}$ is the generator of the group of graph automorphisms. Since $\mathcal{A}_\chi$ is non-cyclic it follows that $\mathcal{A}_\chi= \langle \mathrm{ad}(x) \gamma, F_0^i \rangle$. Therefore, $\mathcal{B}_{\chi}= \langle \gamma, F_0^i \rangle$. By Clifford theory it follows that the character $\chi$ extends to $\mathcal{A}_\chi$ if and only if $\chi$ extends to an $\mathrm{ad}(x) \gamma$-invariant character of $G \langle F_0^i \rangle$. On the other hand, the character $\chi$ extends to $\mathcal{B}_\chi$ if and only if $\chi$ extends to a $\gamma$-invariant character of $G \langle F_0^i \rangle$. We conclude that $\chi$ extends to $G \mathcal{A}_\chi$ if and only if it extends to $G \mathcal{B}_\chi$.
\end{proof}

We have a local version of the previous lemma. Recall that $Q$ is assumed to be a characteristic subgroup of the defect group $D$ of $b$.

\begin{lemma}\label{AEloc}
	Let $\chi \in \Irr(\mathrm{N}_G(Q),B_Q)$. Then the character $\chi$ extends to its inertia group in $\mathrm{N}_{G \mathcal{A}}(Q,B_Q)$ if and only if it extends to its inertia group in $\mathrm{N}_{G \mathcal{B}}(Q,B_Q)$.
\end{lemma}

\begin{proof}
	A short calculation shows that $\mathrm{N}_{G \mathcal{A}}(Q,B_Q)/ \mathrm{N}_G(Q) \cong\mathrm{N}_{\mathcal{A}}(b)$ and similarly $\mathrm{N}_{G \mathcal{B}}(Q,B_Q)/ \mathrm{N}_G(Q)\cong\mathrm{N}_{\mathcal{B}}(b)$. By the same argument as in Lemma \ref{AE} we may assume that the stabilizer $\mathrm{N}_{G \mathcal{A}}(Q)_{\chi}/ \mathrm{N}_G(Q)$ is non-cyclic. Therefore, $\mathcal{A}=\langle \mathrm{ad}(x) \gamma, F_0 \rangle$, where $F_0 \in \langle \phi_0 \rangle$. We conclude that there exist $y,z \in G$ such that $\mathrm{N}_{G \mathcal{A}}(Q)_{\chi}=\mathrm{N}_{G}(Q) \langle y \, \mathrm{ad}(x) \gamma, z F_0^i \rangle$. It follows that $\mathrm{N}_{G \mathcal{B}}(Q)_{\chi}= \mathrm{N}_{G}(Q)\langle yx \gamma, z F_0^i \rangle$. By Clifford theory it follows that the character $\chi$ extends to $\mathrm{N}_{G \mathcal{A}}(Q)_{\chi}$ if and only if $\chi$ extends to a $ y \, \mathrm{ad}(x) \gamma$-invariant character of $\mathrm{N}_{G}(Q) \langle z F_0^i \rangle$. On the other hand, the character $\chi$ extends to $\mathrm{N}_{G \mathcal{B}}(Q)_{\chi}$ if and only if $\chi$ extends to a $y x \gamma$-invariant character of $\mathrm{N}_{G}(Q) \langle F_0^i \rangle$. Therefore, both statements are equivalent.
\end{proof}

\begin{remark}\label{centralizerproblem}
	In Theorem \ref{12} one could try to replace $\mathcal{B}$ by the group $\mathcal{A}$. However, $\mathrm{C}_{\tilde{G} \mathcal{A}}(G)$ could be larger than $\mathrm{Z}(\tilde{G})$, see the remarks following Proposition \cite[Proposition 4.9]{Jordan}. We do not know however how to compute the values of the involved projective representations on this larger group.
\end{remark}

\subsection{The case $D_4$}\label{caseD4}

In the previous section we assumed that $\G^F$ is not of untwisted type $D_4$. The reason for this is that $\G^F$ admits in this case an additional graph automorphism. Thus, many of our considerations have to be altered in order to work in this case. The aim of this section is to provide a certain criterion for the extendibility of characters which is tailored to the situation of Theorem \ref{reductionquasi}.

Suppose in this section only that $\G$ is a simple, simply connected algebraic group of type $D_4$. We let $\phi_0: \G \to \G$ be the field automorphism defined in \cite[4.1]{Jordan} and we consider for any fixed prime power $q=p^f$ the Frobenius endomorphism $F= \phi_0^f : \G \to \G$ defining an $\mathbb{F}_q$-structure such that $\G^F$ is a finite quasi-simple group of untwisted type $D_4$.

%Fix a graph automorphism $\gamma_2: \Gtilde \to \Gtilde$ of order $2$ and $\gamma_3: \Gtilde \to \Gtilde$ a graph automorphism of order $3$.
%
%We consider the subgroup $\mathcal{B}:= \langle \gamma_1,\gamma_2,\phi_0 \rangle$ of $\mathrm{Aut}(\tilde{G})$.

\begin{corollary}\label{explicitD4}
	There exists a subgroup $\mathcal{C}$ of $\mathcal{B}$ such that the image of $\tilde{G} \rtimes \mathcal{C}$ in $\mathrm{Out}(G)$ is the stabilizer of $e_s^{\G^F}$ in $\mathrm{Out}(G)$.
\end{corollary}

\begin{proof}
	Let $\mathrm{Diag}_{\G^F}$ be the image of the group of diagonal automorphisms in $\mathrm{Out}(\G^F)$. The stabilizer of $e_s^{\G^F}$ in $\mathrm{Out}(\G^F)$ contains $\mathrm{Diag}_{\G^F}$ by \cite[Lemma 7.4]{Dat}. Since $\tilde{G} \rtimes \mathcal{B}$ generates all automorphisms of $\G^F$ up to inner automorphisms, there exists a subgroup $\mathcal{C} \leq   \mathcal{B}$ such that the image of $\tilde{G} \rtimes \mathcal{C}$ in $\mathrm{Out}(G)$ is the stabilizer of $e_s^{\G^F}$ in $\mathrm{Out}(G)$.
	%
	%Consider the subgroups $K:= \langle \gamma_1,\gamma_2 \rangle$ and $H:= \langle \phi_0 \rangle$ of $\mathrm{Aut}(\tilde{G})$. By the Lemma of Goursat, the subgroup $U$ has the form	
	%\[U= \{(k,h) \in K_1 \times H_1 \mid \varphi(g K_2)=k H_2 \} ,\]
	%for some $5$-tuple $(K_1,K_2,H_1,H_2,\varphi)$, where $K_2 \lhd K_1 \leq K$ and $H_2 \lhd H_1 \leq H$ and $\varphi: K_1/K_2 \to H_1/H_2$ is a group isomorphism. Since $K \cong S_3$ there exists a graph automorphism $\pi: \G \to \G$ such that $K_1= K_2 \rtimes \langle \pi \rangle$. Let $\phi \in H$ such that $\varphi( \pi K_2)= \phi H_2$. Then it follows that
	%$$U=\langle K_2 \rangle \rtimes \langle \pi \phi \rangle.$$
	%Now, $K_2$ is a subgroup of $K$, so there exist graph automorphisms $\sigma_i: \tilde{\G} \to \tilde{\G}$ with $\sigma_1^2=\sigma_2^3= \mathrm{id}_{\tilde{\G}}$ such that $K_2= \langle \sigma_1,\sigma_2 \rangle$. Moreover, there exist some $i$ with $i \mid f$ such that $(\phi_0^i)|_{\tilde{\G}^F}= \phi$. We let $F_0:= \phi_0^i \pi: \tilde{\G} \to \tilde{\G}$ and we have $F_0^r=F$ for some positive integer $r$.
\end{proof}

%Denote $\mathcal{C}:=\langle \sigma_1,\sigma_2,F_0 \rangle \subseteq \mathrm{Aut}(\tilde{G})$. There exists Frobenius endomorphisms $F_1,F_2: \tilde{\G} \to \tilde{\G}$ such that $F_1^{r_1}=F$ and $F_2^{r_2}=F_2$ for some integers $r_1,r_2$ such that 
Recall that $b$ is a block of $\mathcal{O} \G^F e_s^{\G^F}$ with defect group $D$ and characteristic subgroup $Q$. For every prime $r$ fix a Sylow $r$-subgroup $\mathcal{C}_r$ of $\mathcal{C}$. Note that $\mathcal{C}_r$ is contained in a Sylow $r$-subgroup of $\mathcal{B}$. Hence, there exists a graph automorphism $\gamma_r: \tilde{\G}  \to \tilde{\G}$ of order dividing $r$ and a Frobenius $F_r=\phi_0^{i_r}: \tilde{\G} \to \tilde{\G}$, with $i_r \mid f$, such that $\mathcal{C}_r= \langle \gamma_r, F_r \rangle$. We define $F_0:\tilde{\G} \to \tilde{\G}$ to be the Frobenius endomorphism which is the product of the $F_r$ over all primes $r$ diving the order of $\mathcal{B}$. In the following, we let $I \subseteq \{2,3\}$ be the set such that $\mathcal{C}_i$ is non-cyclic for $i \in I$.

%Consider the group $H:=\mathrm{N}_{G\mathcal{C}}(Q)_{\chi}/ \mathrm{N}_G(Q)$. Since $Q$ is a characteristic subgroup of $D$ we have $H \cong \mathcal{C}$.

\begin{lemma}\label{Sylowreduction}
	Let $\chi \in \Irr( \mathrm{N}_{G \mathcal{B}}(Q)_{\chi} \mid B_Q)$. The character $\chi$ extends to $\mathrm{N}_{G\mathcal{B}}(Q)_{\chi}$ if and only if for every $r \in I$ it extends to $\mathrm{N}_{G \mathcal{R}}(Q)_{\chi}$ for every Sylow $r$-subgroup $\mathcal{R}$ of $\mathcal{C}$.
\end{lemma}

\begin{proof}
	By Corollary \ref{explicitD4}, the image of $\tilde{G} \rtimes \mathcal{C}$ in $\mathrm{Out}(\G^F)$ is the stabilizer of $e_s^{\G^F}$ in $\mathrm{Out}(\G^F)$. Consequently, we have $\mathrm{N}_{G\mathcal{B}}(Q)_{\chi}= \mathrm{N}_{G\mathcal{C}}(Q)_{\chi}$. Denote $H:=\mathrm{N}_{G\mathcal{C}}(Q)_{\chi}/ \mathrm{N}_G(Q)$. There exists a Sylow $r$-subgroup $\mathcal{R}$ of $\mathcal{C}$ such that $\mathrm{N}_{G \mathcal{R}}(Q)_\chi / \mathrm{N}_G(Q)$ is a Sylow $r$-subgroup of $H$. (Note that this property is not necessarily true for all Sylow $r$-subgroups of $\mathcal{C}$.) Moreover, for $r \notin I$ all Sylow $r$-subgroups of $H$ are cyclic. By \cite[Corollary 11.31]{Isaacs} the character $\chi$ extends to $\mathrm{N}_{G\mathcal{C}}(Q)_{\chi}$ if and only if $\chi$ extends to the preimage of a Sylow $r$-subgroup of $H$ for every prime $r$. Since all Sylow $r$-subgroups of $H$ for $r \notin I$ are cyclic it follows that $\chi$ extends to $\mathrm{N}_{G\mathcal{B}}(Q)_{\chi}$ if and only if $\chi$ extends to $\mathrm{N}_{G \mathcal{R}}(Q)_\chi$ for all $r \in I$ and every Sylow $r$-subgroup $\mathcal{R}$ of $\mathcal{C}$.
\end{proof}

Let $\Levi^\ast$ be the minimal Levi subgroup of $\G^\ast$ containing $\mathrm{C}(s)$. By \cite[Lemma 4.6]{Jordan} and \cite[Remark 4.7]{Jordan} there exists a Levi subgroup $\Levi$ of $\G$ in duality with $\Levi^\ast$ such that $\Levi$ is $F_0$-stable. Recall that $\mathcal{O} \G^F e_s^{\G^F}$ and $\mathcal{O} \Levi^F e_s^{\Levi^F}$ are splendid Rickard equivalent, see \cite[Theorem 2.28]{Jordan}. Hence by \cite[Theorem 1.3]{Jordan} we can and we will assume that the defect group $D$ of $b$ is contained in $\Levi^F$.

Let $j \in I$. By the proof of \cite[Proposition 4.9]{Jordan} there exist $x_j \in \G^{F_0}$ such that $\sigma_j:=\mathrm{ad}(x_j) \gamma_j$ stabilizes $\Levi$ and $e_s^{\Levi^F}$. If $j \in \{ 2,3 \} \setminus  I$ there exists some bijective morphism $\pi_j: \tilde{\G} \to \tilde{\G}$ with $\mathcal{C}_j= \langle \pi_j \rangle$ and again by the proof of \cite[Proposition 4.9]{Jordan} there exist $x_j \in \G^{F_0}$ such that $\sigma_j:=\mathrm{ad}(x_j) \pi_j$ stabilizes $\Levi$ and $e_s^{\Levi^F}$. We then define $\mathcal{A}:= \langle \sigma_2,\sigma_3,F_0 \rangle \subseteq \mathrm{Aut}(\tilde{\G}^F)$.

For $r \in I$ consider an arbitrary Sylow $r$-subgroup $\mathcal{R}$ of $\mathcal{C}$. Then we have $\mathcal{R}= \langle \gamma , F_r \rangle$ for some graph automorphism $\gamma \in \mathcal{G}$. Hence, there exists some $x \in \G^{F_0}$ such that $\sigma:=\mathrm{ad}(x) \gamma \in \mathcal{A}$. We then denote $\mathcal{R}_{\mathcal{A}}:= \langle \sigma, F_r \rangle \subseteq \mathrm{Aut}(\tilde{\G}^F)$.

\begin{lemma}\label{LemmaBtoA}
	A character $\chi \in \Irr( \mathrm{N}_{G \mathcal{B}}(Q)_{\chi} \mid B_Q)$ extends to $\mathrm{N}_{G\mathcal{B}}(Q)_\chi$ if and only if for every $r \in I$ it extends to $\mathrm{N}_{G \mathcal{R}_{\mathcal{A}}}(Q)_\chi$ for all Sylow $r$-subgroups $\mathcal{R}$ of $\mathcal{C}$.
\end{lemma}

\begin{proof}
	By Lemma \ref{Sylowreduction}, the character $\chi$ extends to $\mathrm{N}_{G\mathcal{B}}(Q)_{\chi}$ if and only if $\chi$ extends to $\mathrm{N}_{G \mathcal{R}}(Q)_\chi$ for every Sylow $r$-subgroup $\mathcal{R}$ of $\mathcal{C}$ with $r \in I$. Hence it suffices to show that for every such $\mathcal{R}$ the character $\chi$ extends to $\mathrm{N}_{G \mathcal{R}}(Q)_\chi$ if and only if $\chi$ extends to $\mathrm{N}_{G \mathcal{R}_{\mathcal{A}}}(Q)_\chi$. The proof of the latter is now however exactly the same as in Lemma \ref{AEloc}.
\end{proof}

%Recall from \cite[Example 2.3]{Jordan} that $\N_\G(Q)$ is a reductive group with parabolic subgroup $\mathrm{N}_\Para(Q)$ which has a Levi decomposition $\mathrm{N}_{\Para}(Q)=\mathrm{N}_{\Levi}(Q) \C_\U(Q)$. By \cite[Theorem 3.10]{Jordan} the bimodule $H^{d_Q}(\Y_\C_\U(Q)) 
\subsection{Lusztig induction and extension of characters}

	Let $\G$ be a reductive group (not necessarily connected) with Frobenius endomorphism $F: \G \to \G$ and parabolic subgroup $\Para$ with Levi decomposition $\Para= \Levi \ltimes \U$, where $F(\Levi)=\Levi$. Recall that Lusztig induction yields a map $R_{\Levi \subseteq \Para}^\G: \mathbb{Z} \Irr(L) \to \mathbb{Z} \Irr(G)$. In our situation this map will actually be independent of the choice of the parabolic $\Para$ and if this is the case we write $R_{\Levi}^\G$ instead of $R_{\Levi \subseteq \Para}^\G$ 
	
	 Moreover $\mathrm{N}_\G(Q)$ is a reductive group with parabolic subgroup $\mathrm{N}_\Para(Q)$ which has a Levi decomposition $\mathrm{N}_{\Para}(Q)=\mathrm{N}_{\Levi}(Q) \C_\U(Q)$, see \cite[Example 2.3]{Jordan}. Consequently, we also obtain a map $R_{\mathrm{N}_{\Levi}(Q) \subseteq \mathrm{N}_{\Para}(Q)}^\G: \mathbb{Z} \Irr(\mathrm{N}_{L}(Q)) \to \mathbb{Z} \Irr(\mathrm{N}_{G}(Q))$.

\begin{lemma}\label{lemma1}
	Suppose that $\G$ is connected reductive and let $\iota: \G \hookrightarrow \Gtilde$ be a regular embedding.
	Let $\chi \in \Irr(\tilde{L})$, $\lambda \in \Irr(\tilde{G}/G)$ and $\psi \in \Irr(\mathrm{N}_{\tilde{L}}(Q))$. Then we have
	\begin{enumerate}[label=(\alph*)]
		\item 
		$\lambda R_{\tilde{\Levi} \subseteq \tilde{\Para}}^{\tilde{\G}}(\chi)= R_{\tilde{\Levi} \subseteq \tilde{\Para}}^{\tilde{\G}}( \mathrm{Res}^{\tilde{G}}_{\tilde{L}}(\lambda) \chi)$,
		\item 
		$\mathrm{Res}^{\tilde{G}}_{\mathrm{N}_{\tilde{G}}(Q)}(\lambda) R_{\mathrm{N}_{\tilde{\Levi}}(Q) \subseteq \mathrm{N}_{\tilde{\Para}}(Q)}^{\mathrm{N}_{\tilde{\G}}(Q)}( \psi)=  R_{\mathrm{N}_{\tilde{\Levi}}(Q) \subseteq \mathrm{N}_{\tilde{\Para}}(Q)}^{\mathrm{N}_{\tilde{\G}}(Q)}( \mathrm{Res}^{\tilde{G}}_{\mathrm{N}_{\tilde{L}}(Q)}(\lambda) \psi)$.
	\end{enumerate}
\end{lemma}

\begin{proof}
	Part (a) is classical. It is proved in \cite[Proposition 13.30(ii)]{DM} using the character formula for Lusztig induction from \cite[Proposition 12.2]{DM}.  
	
	The character formula for Lusztig induction has been generalized to disconnected reductive groups, see \cite[Proposition 2.6(i)]{DM2}. (Note that our definition of Levi subgroups and parabolic subgroups in \cite[2.1]{Jordan} is more general than the one in \cite{DM2}, but the same proof applies to our set-up.) Applying this explicit character formula to $R_{\mathrm{N}_{\tilde{\Levi}}(Q)}^{\mathrm{N}_{\tilde{\G}}(Q)}( \psi)$ gives the result in (b).
	
	A more conceptual proof of the lemma can be obtained as follows. Consider the $\mathrm{N}_{G}(Q) \times \mathrm{N}_{L}(Q)^{\mathrm{opp}} \Delta(\mathrm{N}_{\tilde L}(Q))$-module $M:=H_c^i(\Y_{\C_\U(Q)}^{\mathrm{N}_\G(Q)},K)$. By \cite[Lemma 2.7]{Jordan} we have 
	$$\Ind_{\mathrm{N}_{G}(Q) \times \mathrm{N}_{L}(Q)^{\mathrm{opp}} \Delta(\mathrm{N}_{\tilde L}(Q))}^{\mathrm{N}_{\tilde G}(Q) \times \mathrm{N}_{\tilde L}(Q)^{\mathrm{opp}}} M \cong \tilde{M},$$ where $\tilde{M}:=H_c^i(\Y_{\C_\U(Q)}^{\mathrm{N}_{\tilde{\G}}(Q)},K)$. Consider now the character $\mu: \mathrm{N}_{\tilde{G}}(Q) \times  \mathrm{N}_{\tilde{L}}(Q)^{\mathrm{opp}} \to K^\times,(\tilde{g},\tilde{l}) \mapsto \lambda(\tilde{g}) \lambda(\tilde{l}),$ and observe that $\mu$ is trivial on $\Delta \mathrm{N}_{\tilde{L}}(Q)$. Therefore, we have $\tilde{M} \cong \tilde{M} \otimes_K \mu$. From this the claim follows easily by definition of Lusztig induction.
\end{proof}

%\begin{lemma}\label{lemma2}
%Let $\G$ be a reductive group and $\mathcal{Y}$ a rational series of $\G^F$. Let $(T,\theta) \in \mathcal{Y}$. Then for every character $\chi \in \mathcal{E}(\G^F,e_{\mathcal{X}})$ we have $\chi|_{\mathrm{Z}(\G^\circ)^F}= \chi(1) \theta|_{\mathrm{Z}(\G^\circ)^F}$.
%\end{lemma}
%
%\begin{proof}
%
%\end{proof}

\begin{lemma}\label{lemma2}
For every $\chi \in \Irr(\Levi^F)$ the characters $\chi$ and $R_{\Levi \subseteq \Para}^{\G}(\chi)$ restrict to multiples of the same central character on $Z:= (\mathrm{Z}(\G) \cap \Levi)^F$.
\end{lemma}

\begin{proof}
	The diagonal action of $Z$ fixes the variety $\Y_\U^\G$ pointwise. Hence the diagonal action of $Z$ on the bimodule $H^{i}_c( \Y_\U^\G, K)$ for every $i$ is trivial. The claim of the lemma follows from this.
\end{proof}

\subsection{A first reduction of the iAM-condition}

In this section we describe the first step to reducing the verification of the iAM-condition. Before stating the main theorem we state the following theorem which will be crucial for our reduction.

\begin{theorem}\label{starcondition}
	In every $\tilde{G}$-orbit of $\Irr(G, e_s^{G})$ there exists a character $\chi \in \Irr(G, e_s^{G})$ satisfying assumption (i) of Theorem \ref{12}.
\end{theorem}

\begin{proof}
For groups of Lie type not of type $D$ the statement is proved in \cite[Theorem B]{Sternbedingung}. For groups of type $D$ this was recently proved in \cite{TypeD}.
%	
%	and conjectured to hold as well in this type. Therefore, Theorem \ref{intromain} from the introduction is a consequence of Theorem \ref{maintheorem}.	
\end{proof}

In the following we abbreviate $M_L:= \mathrm{N}_L(Q)$, $\hat{M}:=\mathrm{N}_{G \mathcal{A}}(Q)$, $\hat{M}_L:=\mathrm{N}_{L \mathcal{A}}(Q)$ and $\tilde{M}_L:= \mathrm{N}_{\tilde{L}}(Q)$.

\begin{theorem}\label{reductionquasi}
	Let $b$ be a block idempotent of $\mathrm{Z}(\mathcal{O} G e_s^{G})$ and $c \in \mathrm{Z}(\mathcal{O} L e_s^{L})$ the block idempotent corresponding to $b$ under the Morita equivalence between $\mathcal{O} L e_s^{L}$ and $\mathcal{O} G e_s^{G}$ given by $H^{\mathrm{dim}(\Y_\U^\G)}_c( \Y_\U^\G, \mathcal{O}) e_s^L$. Assume that the following hold.
	%a $\mathrm{N}_{\tilde{L} \mathcal{A}}(Q,b_Q)$-equivariant bijection $\tilde{\varphi}: \Irr_0(c) \to \Irr_0(C_Q)$
	\begin{enumerate}[label=(\roman*)]

		\item There exists an $\Irr( \tilde{M}_L / M_L) \rtimes \hat{M}_L$-equivariant bijection $\tilde{\varphi}: \Irr(\tilde{L} \mid \Irr_0(c) ) \to \Irr(\tilde{M}_L \mid \Irr_0(C_Q))$ such that it maps characters covering the character $\nu \in \Irr(\mathrm{Z}(\tilde{G}))$ to a character covering $\nu$.
		%CHANGED FROM Z(\tilde{L}) to Z(\tilde{G})
		\item There exists an $\mathrm{N}_{\tilde{L} \mathcal{A}}(Q,C_Q)$-equivariant bijection 
		$\varphi:  \Irr_0(L , c) \to \Irr_0(M_L , C_Q)$ which satisfies the following two conditions:
		\begin{itemize}
			\item 
			If $\chi \in \Irr_0(L,c)$ extends to a subgroup $H$ of $L \mathcal{A}$ then $\varphi(\chi)$ extends to $\mathrm{N}_{H}(Q)$.
			\item $\tilde{\varphi}(\Irr(\tilde{L} \mid \chi))= \Irr(\tilde{M} \mid \varphi(\chi))$ for all $\chi \in \Irr_0(c)$.
		\end{itemize}
		%Clifford correspondent characters lie in Harris-Kn\"o,rr correspondent blocks: 
		\item For every $\theta \in \Irr_0(c)$ and $\tilde{\theta} \in \Irr(\tilde{L} \mid \theta)$ the following holds: If $\theta_0 \in \Irr(\tilde{L}_\theta \mid \theta)$ is the Clifford correspondent of $\tilde{\theta} \in \Irr(\tilde{L})$ then $\mathrm{bl}(\theta_0)=\mathrm{bl}(\theta_0')^{\tilde{L}_\theta}$, where $\theta'_0 \in \Irr(\tilde{M}_{\varphi(\theta)} \mid \varphi(\theta))$ is the Clifford correspondent of $\tilde{\varphi}(\tilde{\theta})$.
		
	\end{enumerate}
	Then the block $b$ is iAM-good.
	%$$(( \tilde{G} \mathcal{B})_\chi, G, \chi) \geq_b ((\tilde{M} \mathrm{N}_{G\mathcal{B}}(Q))_{\Psi(\chi)},M, \Psi(\chi)).$$
\end{theorem}

\begin{proof}
	By \cite[Theorem 1.5(2)]{Broue3} and \cite[Theorem 3.1]{Broue3} it follows that derived equivalences between blocks of group algebras induce character bijections which preserve the height of corresponding characters. Hence, by \cite[Theorem 3.10]{Jordan} we obtain bijections 
	$$R_{\Levi}^{\G}: \Irr_0(L,c) \to \Irr_0(G,b) \text{ and } R_{\mathrm{N}_{\Levi}(Q)}^{\mathrm{N}_{\G}(Q)}: \Irr_0(\mathrm{N}_L(Q),C_Q) \to \Irr_0(\mathrm{N}_G(Q),B_Q).$$
	The inverse of these bijections is then denoted by ${}^\ast R_{\Levi}^{\G}$ and ${}^\ast R_{\mathrm{N}_{\Levi}(Q)}^{\mathrm{N}_{\G}(Q)}$ respectively.

	We define $\Psi:\Irr_0(G,b) \to \Irr_0(\mathrm{N}_G(Q), B_Q)$ to be the bijection which makes the following diagram commutative:
	\begin{center}
		\begin{tikzpicture}
		\matrix (m) [matrix of math nodes,row sep=3em,column sep=4em,minimum width=2em] {
			
			\Irr_0(L, c) & \Irr_0(G, b)  \\
			\Irr_0(M_L,C_Q) & \Irr_0(M,B_Q) 
			\\};
		\path[-stealth]
		(m-1-2) edge node [left] {$\Psi$} (m-2-2)
		(m-1-1) edge node [above] {$R_{\Levi}^{\G}$} (m-1-2)
		(m-2-1) edge node [above] {$R_{\mathrm{N}_{\Levi}(Q)}^{\mathrm{N}_{\G}(Q)}$} (m-2-2)
		(m-1-1) edge node [left] {$\varphi$} (m-2-1);
		
		\end{tikzpicture}
	\end{center}
	Note that the bijection $R_{\Levi}^{\G}: \Irr_0(L,c) \to \Irr_0(G,b)$ is $\mathrm{N}_{\tilde{L} \mathcal{A}}(c)$-equivariant and the bijection $R_{\mathrm{N}_{\Levi}(Q)}^{\mathrm{N}_{\G}(Q)}: \Irr_0(\mathrm{N}_{L}(Q),C_Q) \to \Irr_0(\mathrm{N}_{G}(Q),B_Q)$ is $\mathrm{N}_{\tilde{L} \mathcal{A}}(Q,C_Q)$-equivariant by \cite[Lemma 2.23]{Jordan}. We see that $\mathrm{N}_G(Q) \mathrm{N}_{\tilde{L} \mathcal{A}}(Q,C_Q)= \mathrm{N}_{\tilde{G} \mathcal{A}}(Q,B_Q)$ and  consequently the bijection $\Psi:\Irr_0(G,b) \to \Irr_0(M,B_Q)$ is $\mathrm{N}_{\tilde{G} \mathcal{A}}(Q,B_Q)$-equivariant.
	
	Fix a character $\chi \in \Irr_0(G, b)$. For $l \in {\tilde{L}}$ we note that the block ${}^l b$ of $\mathcal{O} \G^F e_s^{\G^F}$ also satisfies the assumptions of the theorem with the map $\varphi$ replaced by $\varphi': \Irr_0(L,{}^l c) \to \Irr_0(\mathrm{N}_L({}^l Q),{}^l C_Q)$ given by $\varphi'(\theta)= {}^l \varphi({}^{l^{-1}} \theta)$ for $\theta \in \Irr_0(L,{}^l c)$.
	%As before we obtain a map $\Psi':\Irr_0({}^g b) \to \Irr_0({}^g B_Q)$ which satisfies $\Psi'(\theta)= {}^g \varphi({}^{g^{-1}} \theta)$ for $\theta \in \Irr_0(G,{}^g c)$.
	Using Theorem \ref{starcondition} we can, by possibly taking a $\tilde{G}$-conjugate of $b$, assume that the character $\chi$ satisfies assumption (i) of Theorem \ref{12}. We denote $\chi' := \Psi(\chi)$ and show that the characters $\chi$ and $\chi'$ satisfy the conditions of Theorem \ref{12}.
	
	Since the bijection $\Psi:\Irr_0(G,b) \to \Irr_0(M,B_Q)$ is $\mathrm{N}_{\tilde{G} \mathcal{A}}(Q,B_Q)$-equivariant we deduce that condition (iii) in Theorem \ref{12} is satisfied and we have 
	$$( \tilde{M} \mathrm{N}_{G \mathcal{B}}(Q))_{\chi'}= \tilde{M}_{\chi'} \mathrm{N}_{G \mathcal{B}}(Q)_{\chi'}.$$
	
	To show condition (ii) in Theorem \ref{12} let us first assume that $\G^F$ is not of type $D_4$. Since Assumption \ref{starcondition} holds for the character $\chi$ it follows by Lemma \ref{AE} that the character $\chi$ extends to its inertia group in $G \mathcal{A}$. It follows by \cite[Theorem 5.8]{Jordan} that ${}^\ast R_L^G(\chi)$ extends to its inertia group in $L \mathcal{A}$. By assumption (ii), the character $\varphi({}^\ast R_{\Levi}^{\G}(\chi))$ extends to its inertia group in $\hat{M}_L$. Hence, by \cite[Theorem 5.11]{Jordan} the character $\chi'$ extends to $\hat{M}_{\chi'}$. Now Lemma \ref{AEloc} shows that condition (ii) in Theorem \ref{12} is satisfied.
	
	Now assume that $\G^F$ is of type $D_4$. We use the notation of Section \ref{caseD4} and show the following result.
	
	\begin{proposition}\label{forThm618}
		A character $\chi \in  \Irr( \mathrm{N}_{G \mathcal{B}}(Q)_{\chi} \mid B_Q)$ extends to $\mathrm{N}_{G\mathcal{B}}(Q)_\chi$ if and only if for all $r \in I$ the character $\psi:={}^\ast R_{\mathrm{N}_{\Levi}(Q)}^{\mathrm{N}_{\G}(Q)}( \chi)$ extends to $\mathrm{N}_{L \mathcal{R}_{\mathcal{A}}}(Q)_\psi$ for every Sylow $r$-subgroup $\mathcal{R}$ of $\mathcal{C}$.
	\end{proposition}
	
	\begin{proof}
		By Lemma \ref{LemmaBtoA} the character $\chi$ extends to $\mathrm{N}_{G\mathcal{B}}(Q)_\chi$ if and only if $\chi$ extends for every $r\in I$ to $\mathrm{N}_{G \mathcal{R}_{\mathcal{A}}}(Q)_\chi$ for every Sylow $r$-subgroup $\mathcal{R}$ of $\mathcal{C}$.
		
		Now for $r \in I$ recall that $\mathcal{R}_{\mathcal{A}}= \langle \sigma, F_r \rangle$ for some bijective morphism $\sigma$ of $\Gtilde$. Thus, we can apply \cite[Theorem 5.11]{Jordan} to the commuting automorphisms $\sigma: \Gtilde \to \Gtilde$ and $F_r: \Gtilde \to \Gtilde$. By \cite[Remark 1.8(b)]{Jordan} extendibility of characters is preserved in this situation. This implies that the character $\chi$ extends to $\mathrm{N}_{G \mathcal{R}_{\mathcal{A}}}(Q)_\chi$ if and only if $\psi$ extends to $\mathrm{N}_{L \mathcal{R}_{\mathcal{A}}}(Q)_\psi$.
		%Applying Theorem \ref{loc} again, but this time to the automorphisms $\sigma_2: \G \to \G$ and $F_2: \G \to \G$ yields that $\chi$ extends to $\mathrm{N}_{G (\mathcal{A}_2)}(Q)_\chi$ if and only if $\psi$ extends to $\mathrm{N}_{L (\mathcal{A}_2)}(Q)_\psi$. This shows the claim of the proposition.
	\end{proof}

	 Since the character $\chi$ satisfies assumption (i) of Theorem \ref{12} we know by Proposition \ref{forThm618} that $\psi:={}^\ast R_{L}^{G}( \chi)$ extends to $(G \mathcal{R}_{\mathcal{A}})_\psi$ for every Sylow $r$-subgroup $\mathcal{R}$ of $\mathcal{C}$ with $r \in I$. By assumption (ii), it follows that $\varphi(\psi)$ extends to $\mathrm{N}_{L \mathcal{R}_{\mathcal{A}}}(Q)_\psi$ for every Sylow $r$-subgroup $\mathcal{R}$ of $\mathcal{C}$ with $r \in I$. Applying Proposition \ref{forThm618} again yields that $\chi'$ extends to its inertia group in $\mathrm{N}_{G \mathcal{B}}(Q)$. Thus, condition (ii) in Theorem \ref{12} is also satisfied in this case.
	
	To verify condition (v) of Theorem \ref{12} we show the following lemma.
	%$$(( \tilde{G} \mathcal{B})_\chi, G, \chi) \geq_c ((\tilde{M} \mathrm{N}_{G\mathcal{B}}(Q))_{\Psi(\chi)},M, \Psi(\chi)).$$

	\begin{lemma}\label{intermediatesubgroup}
		There exist characters $\tilde{\chi}_0 \in \Irr(\tilde{G}_\chi \mid \chi)$ and $\tilde{\chi}'_0 \in \Irr(\tilde{M}_{\chi'} \mid \chi')$ satisfying $\mathrm{bl}(\tilde{\chi}'_0)^{\tilde{G}_{\chi}}=\mathrm{bl}(\tilde{\chi}_0)$.
	\end{lemma}
	%Before we continue proving that the assumptions in Theorem \ref{12} are satisfied we need to construct and analyse some equivalences relative to the inertia group of $\chi$ in $\tilde{G}$.
	\begin{proof}
		Let $\tilde{\chi}_0$ be a character of $J:=\tilde{G}_{\chi}$ extending $\chi$. Let $J_L$ be the subgroup of $\tilde{L}/L$ corresponding to $J$ under the natural isomorphism $\tilde{L} / L \cong \tilde{G} / G$.

		Recall that $C:=G \Gamma_c( \Y_\U^\G,\mathcal{O})^{\mathrm{red}}e_s^{\Levi^F}$ is a complex of $\mathcal{O}[(G \times L^{\mathrm{opp}}) \Delta(\tilde{L})]$-modules. Moreover by \cite[Proposition 1.1]{Godement}, we have a canonical isomorphism
		$$\mathrm{Ind}_{( G \times L^{\mathrm{opp}}) \Delta (\tilde{L})}^{\tilde{G} \times \tilde{L}^{\mathrm{opp}}} G \Gamma_c( \Y_\U^\G,\mathcal{O}) e_s^{L} \cong G \Gamma_c( \Y_\U^{\tilde{\G}},\mathcal{O}) e_s^L$$
		in $\mathrm{Ho}^b(\mathcal{O} [\tilde{G} \times \tilde{L}^{\mathrm{opp}}])$. By \cite[Lemma 7.4]{Dat}, the complex $G \Gamma_c( \Y_\U^{\tilde{\G}},\mathcal{O}) e_s^L$ induces a splendid Rickard equivalence between $\mathcal{O} \tilde{G} e_s^{G}$ and $\mathcal{O} \tilde{L} e_s^{L}$.
		% Give good argument why this is still a Rickard equivalence
		Thus, the complex $\tilde{C}:=\mathrm{Ind}_{( G \times L^{\mathrm{opp}}) \Delta(J_L)}^{J \times J_L^{\mathrm{opp}}} (G \Gamma_c( \Y_\U^\G,\mathcal{O})^{\mathrm{red}}) c$ induces a splendid Rickard equivalence between $\mathcal{O} J b$ and $\mathcal{O} J_L c$, see \cite[Lemma 1.9]{Jordan}. Denote $\tilde{b}:= \mathrm{bl}(\tilde{\chi})$ and let $\tilde{c}$ be the block corresponding to $\tilde{b}$ under the Rickard equivalence induced by $\tilde{C}$. The cohomology of $C$ is concentrated in degree $d:= \mathrm{dim}(\Y_\U^\G)$ and $H^d(\tilde{C}) \cong  \mathrm{Ind}_{( G \times L^{\mathrm{opp}}) \Delta J_L}^{J \times J_L^{\mathrm{opp}}} H^d(C)$. By \cite[Theorem 1.7]{Jordan} the bimodule $H^d(\tilde{C})$ induces a Morita equivalence between $\mathcal{O} J_L c$ and $\mathcal{O} J b$. We denote by 
		$$R:\Irr(J_L , c) \to \Irr(J , b)$$
		the associated bijection between irreducible characters and its inverse by ${}^\ast R$.
		
		%Denote $\tilde{c}:=\tilde{c}_0^{\tilde{L}}$.
		
		The complex $\mathrm{Ind}_{(\mathrm{C}_{J}(Q) \times \mathrm{C}_{J_L}(Q)^{\mathrm{opp}}) \Delta( \mathrm{N}_{J_L}(Q))}^{\mathrm{N}_{J}(Q) \times \mathrm{N}_{J_L}(Q)^{\mathrm{opp}}}(\mathrm{Br}_{\Delta Q}(\tilde{C} ))$ induces a derived equivalence between the algebras $k \mathrm{N}_{J}(Q) B_Q$ and $k \mathrm{N}_{J_L}(Q) C_Q$, see \cite[Proposition 1.12]{Jordan}. Denote
		$$\tilde{C}_{\mathrm{loc}}:=\mathrm{Ind}_{(\mathrm{C}_G(Q) \times \mathrm{C}_L(Q)^{\mathrm{opp}}) \Delta( \mathrm{N}_{\tilde{L}}(Q))}^{( \mathrm{N}_{J}(Q) \times \mathrm{N}_{J_L}(Q)^{\mathrm{opp}} )\Delta( \mathrm{N}_{\tilde{L}}(Q)) }  (G \Gamma_c( \Y_{\C_\U(Q)}^{\mathrm{C}_\G(Q)}, \mathcal{O})) C_Q.$$
		%%
		%SOME MORE STUFF
		As in the proof of \cite[Lemma 1.20]{Jordan} it follows that 
		$$\tilde{C}_{\mathrm{loc}} \otimes_{\mathcal{O}} k \cong \mathrm{Ind}_{\mathrm{C}_{J}(Q) \times \mathrm{C}_{J_L}(Q)^{\mathrm{opp}} \Delta( \mathrm{N}_{J_L}(Q))}^{\mathrm{N}_{J}(Q) \times \mathrm{N}_{J_L}(Q)^{\mathrm{opp}}}(\mathrm{Br}_{\Delta Q}(\tilde{C} )) C_Q$$
		in $\mathrm{Ho}^b(k[ \mathrm{N}_{J}(Q) \times \mathrm{N}_{J_L}(Q)^{\mathrm{opp}}])$.
		
		%Since this complex is splendid (WHY IS IT SPLENDID AGAIN?) there exists a a unique complex $\tilde{C}_{\mathrm{loc}}$ of $\mathcal{O}[ \mathrm{N}_{J_L}(Q) \times \mathrm{N}_J(Q)^{\mathrm{opp}} \Delta( \mathrm{N}_{J_L}(Q))]$-modules lifting it, see \ref{LiftingSR}. 
		
		The cohomology of $G \Gamma_c( \Y_{\C_\U(Q)}^{\C_\G(Q)}, \mathcal{O}) \mathrm{br}_Q(e_s^{\Levi^F})$ is concentrated in degree $d_Q:=\mathrm{dim}(\Y_{\C_\U(Q)}^{\C_\G(Q)})$ by \cite[Proposition 4.11]{Dat}. Moreover, the bimodule $H_c^{d_Q}(\Y_{\C_\U(Q)}^{\C_\G(Q)},\mathcal{O})c_Q$ induces a Morita equivalence between $\mathcal{O}\C_{L}(Q) c_Q$ and $\mathcal{O}\C_{G}(Q) b_Q$. Now \cite[Lemma 1.11]{Jordan} together with \cite[Corollary 1.6]{Jordan} imply that $H^{d_Q}(\tilde{C}_{\mathrm{loc}})$ induces a Morita equivalence between $\mathcal{O} \mathrm{N}_{J_L}(Q) C_Q$ and $\mathcal{O} \mathrm{N}_{J}(Q)B_Q$. We denote the associated character bijection by $R_{\mathrm{loc}}:\Irr(\mathrm{N}_{J_L}(Q) , C_Q) \to \Irr(\mathrm{N}_J(Q), B_Q)$ and its inverse by ${}^\ast R_{\mathrm{loc}}:\Irr(\mathrm{N}_{J}(Q) , B_Q) \to \Irr(\mathrm{N}_{J_L}(Q), C_Q)$. Let $\tilde{\chi}= \mathrm{Ind}_{J}^{\tilde{G}}(\tilde{\chi}_0)$ and define
		$$\tilde{\chi}':=R_{\mathrm{N}_{\tilde{\Levi}}(Q)}^{\mathrm{N}_{\tilde{\G}}(Q)} \circ \tilde{\varphi} \circ {}^\ast R_{\tilde{\Levi}}^{\tilde{\G}} (\tilde{\chi}).$$
		By construction, $\tilde{\chi}' \in \Irr( \mathrm{N}_{\tilde{G}}(Q) \mid \chi')$. We let $\tilde{\chi}_0' \in \Irr( \tilde{M}_{\chi'})$ be the unique character covering $\chi'$ with $\mathrm{Ind}_{\tilde{M}_{\chi'}}^{\tilde{M}}(\tilde{\chi}'_0)= \tilde{\chi}'_0$. Let $\theta:={}^\ast R_{\Levi}^{\G}(\chi)$ and $\tilde{\theta}:= {}^\ast R_{\tilde{\Levi}}^{\tilde{\G}}(\tilde{\chi})$. We have
		$$\mathrm{Ind}_{\mathrm{N}_{J}(Q) \times \mathrm{N}_{J_L}(Q)^{\mathrm{opp}} \Delta( \mathrm{N}_{\tilde{L}}(Q))}^{\mathrm{N}_{\tilde{G}}(Q) \times \mathrm{N}_{\tilde{L}}(Q)^{\mathrm{opp}}} H^{d_Q}( \tilde{C}_{\mathrm{loc}}) \cong H_c^{\mathrm{dim}}(\Y_{\mathrm{N}_\U(Q)}^{\mathrm{N}_{\tilde{\G}}(Q)}, \mathcal{O}) \mathrm{Tr}_{\mathrm{N_{J_L}(Q)}}^{\mathrm{N}_{\tilde{L}}(Q)}(C_Q).$$ 
		%By Remark \ref{characters} we therefore conclude that the equality $$R_{\mathrm{loc}} \circ \mathrm{Res}^{\mathrm{N}_{\tilde{L}}(Q)}_{\mathrm{N}_{J_L}(Q)} =\mathrm{Res}^{\mathrm{N}_{\tilde{G}}(Q)}_{\mathrm{N}_{J}(Q)} \circ  R_{\mathrm{N}_{\tilde{L}}(Q)}^{\mathrm{N}_{\tilde{G}}(Q)}$$ holds on the set $\Irr( \mathrm{N}_{\tilde{G}}(Q), \Tr_{\mathrm{N}_{J}(Q)}^{\mathrm{N}_{\tilde{G}}(Q)}(B_Q))$. By taking adjoint functors we deduce that $$\mathrm{Ind}^{\mathrm{N}_{\tilde{L}}(Q)}_{\mathrm{N}_{J_L}(Q)} \circ {}^\ast R_{\mathrm{loc}}={}^\ast R_{\mathrm{N}_{\tilde{L}}(Q)}^{\mathrm{N}_{\tilde{G}}(Q)} \circ \mathrm{Ind}^{\mathrm{N}_{\tilde{G}}(Q)}_{\mathrm{N}_{J}(Q)}.$$
		Hence, by \cite[Remark 1.8]{Jordan} we obtain
		$$\mathrm{Ind}_{\mathrm{N}_{J_L}(Q)}^{\mathrm{N}_{\tilde{L}}(Q)}({}^\ast R_{\mathrm{loc}}(\tilde{\chi}_0')) =
		{}^\ast R_{\mathrm{N}_{\tilde{\Levi}}(Q)}^{\mathrm{N}_{\tilde{\G}}(Q)}(\Ind_{\mathrm{N}_J(Q)}^{\mathrm{N}_{\tilde{G}}(Q)}( \tilde{\chi}_0'))=
		{}^\ast R_{\mathrm{N}_{\tilde{\Levi}}(Q)}^{\mathrm{N}_{\tilde{\G}}(Q)}(\tilde{\chi}')= \tilde{\varphi}({}^\ast R_{\tilde{\Levi}}^{\tilde{\G}}( \tilde{\chi}))= \tilde{\theta}.$$
		Thus, ${}^\ast R_{\mathrm{loc}}(\tilde{\chi}_0')\in \Irr(\mathrm{N}_{J_L}(Q) \mid \varphi(\theta))$ is the Clifford correspondent of $\tilde{\varphi}(\tilde{\theta})$. Consequently, we have 
		$$\mathrm{bl}({}^\ast R_{\mathrm{loc}}(\tilde{\chi'_0}))^{J_L}=\mathrm{bl}({}^\ast R(\tilde{\chi}_0)) $$
		by assumption (iii).
%		In other words, $\mathrm{bl}({}^\ast R_{\mathrm{loc}}(\tilde{\chi'_0}))$ is the Harris--Kn\"orr correspondent of $\mathrm{bl}({}^\ast R(\tilde{\chi}_0))$ in the sense of Corollary \ref{HKCoro}.
		
		Moreover, the bimodule $H^{d_Q}(\tilde{C}_{\mathrm{loc}}) \mathrm{bl}({}^\ast R_{\mathrm{loc}}(\tilde{\chi}'_0))$ induces a Morita equivalence between the blocks $\mathcal{O} \mathrm{N}_{J}(Q) \mathrm{bl}( \tilde{\chi}'_0)$ and $\mathcal{O} \mathrm{N}_{J_L}(Q) \mathrm{bl}({}^\ast R_{\mathrm{loc}}(\tilde{\chi}'_0))$. On the other hand, as shown above, we have
		$$\tilde{C}_{\mathrm{loc}} \otimes_{\mathcal{O}} k \cong \mathrm{Ind}_{\mathrm{C}_{J}(Q) \times \mathrm{C}_{J_L}(Q)^{\mathrm{opp}} \Delta( \mathrm{N}_{J_L}(Q))}^{\mathrm{N}_{J}(Q) \times \mathrm{N}_{J_L}(Q)^{\mathrm{opp}}}(\mathrm{Br}_{\Delta Q}(\tilde{C} ))$$
		in $\mathrm{Ho}^b(k[ \mathrm{N}_{J}(Q) \times \mathrm{N}_{J_L}(Q)^{\mathrm{opp}}])$.  Since $\mathrm{bl}({}^\ast R_{\mathrm{loc}}(\tilde{\chi'_0}))$ is the Harris--Kn\"orr correspondent of $\mathrm{bl}({}^\ast R(\tilde{\chi}_0))$, it follows by \cite[Remark 1.21]{Jordan} that $\mathrm{bl}(\tilde{\chi}'_0)$ is the Harris--Kn\"orr correspondent of $\mathrm{bl}(\tilde{\chi}_0)$. In other words we have $\mathrm{bl}(\tilde{\chi}'_0)^J=\mathrm{bl}(\tilde{\chi}_0)$.
	\end{proof}
	
	We now complete the proof of Theorem \ref{reductionquasi}.
	Lemma \ref{intermediatesubgroup} together with Lemma \ref{Blockinduction2} implies that condition (v) in Lemma \ref{12} is satisfied.
	
	From Lemma \ref{lemma1} and assumption (i) it follows that the first part of condition (iv) in Theorem \ref{12} is satisfied. Moreover, Lemma \ref{lemma2} and assumption (i) imply that the second part of condition (iv) in Theorem \ref{12} is satisfied. Therefore, all conditions in Theorem \ref{12} are satisfied. Consequently, we have
	$$(( \tilde{G} \mathcal{B})_\chi /Z, G/Z, \overline{\chi}) \geq_b ((\tilde{M} \mathrm{N}_{G \mathcal{B}}(Q))_{\chi'} / Z,\mathrm{N}_G(Q), \overline{\chi'}),$$
	where $\mathrm{Z}:= \mathrm{Z}(G) \cap \mathrm{Ker}(\chi)$.
	By the Butterfly Theorem, see Theorem \ref{Butterfly}, it follows that $\Psi:\Irr_0(G,b) \to \Irr_0(M,B_Q)$ is a strong iAM-bijection. Consequently, the block $b$ is iAM-good.
	%By Lemma \ref{Blockinduction2} we deduce that
	% $$(( \tilde{G} \mathcal{B})_\chi /Z, G/Z, \overline{\chi}) \geq_b ((\tilde{M} \mathrm{N}_{G \mathcal{B}}(Q))_{\chi'} / Z,M, \overline{\chi'}).$$
	%By \cite[Theorem 9.14]{NavarroBook} there exist unique blocks $\tilde{b}_0$ of $\tilde{G}_{b}$ and $\tilde{b}'_0$ of $\tilde{M}_{B_Q}$ such that $\tilde{b}_0^{\tilde{G}}= \tilde{b}$ and $(\tilde{b}'_0)^{\tilde{M}}= \tilde{b}'$.
	%By assumption (d), the character $\tilde{\varphi} {}^\ast R_{\tilde{L}}^{\tilde{G}}(\tilde{\chi})$ lies in the unique block which induces up to $\tilde{c}$. Thus, it lies in the block $\tilde{c}'$. Since $R_{\mathrm{N}_{\tilde{L}}(Q)}^{\mathrm{N}_{\tilde{G}}(Q)}$ gives a bijection $R_{\mathrm{N}_{\tilde{L}}(Q)}^{\mathrm{N}_{\tilde{G}}(Q)}: \Irr(\mathrm{N}_{\tilde{L}}(Q) \mid \tilde{c}' )$ it follows that $\tilde{\chi}'$ is in the block $\tilde{b}'$. By definition of $\tilde{b}'$ we have $(\tilde{b}')^{\tilde{G}}=\tilde{b}$. Thus, condition (b) of Proposition \ref{Blockinduction} is satisfied and we have
	%$$(( \tilde{G} \mathcal{A})_\chi, G, \chi) \leq_b ((\tilde{M} \hat{M})_{\Psi(\chi)},M, \Psi(\chi)).$$
	% Let $\chi_L \in \Irr_0(c)$ such that $R_L^G(\chi_L)= \chi$ and let $\chi'_L \in \Irr_0(C_Q)$ with $R_{\mathrm{N}_L(Q)}^{\mathrm{N}_G(Q)}(\chi'_L)=\chi'$. Let $\tilde{\chi}_L \in \Irr(\tilde{L} \mid \chi_L$. S
\end{proof}

%
%\newpage

\section{Reduction to quasi-isolated blocks}

\subsection{Strictly quasi-isolated blocks}\label{sectionquasi}

Let $\mathbf{H}$ be a connected reductive group with Frobenius $F: \mathbf{H} \to \mathbf{H}$ and dual group $\mathbf{H}^\ast$. Recall that an element $t \in \mathbf{H}^\ast$ is called \textit{quasi-isolated} if $\mathrm{C}_{\G^\ast}(t)$ is not contained in a proper Levi subgroup of $\mathbf{H}^\ast$. In view of Theorem \ref{BDRintro} it makes sense to consider a slightly weaker notion:

%Recall that a semisimple element $s \in (\G^\ast)^F$ is \textit{quasi-isolated} if $\mathrm{C}_{\G^\ast}(s) \mathrm{C}_{\G^\ast}^\circ(s)$ is not contained in a proper Levi subgroup of $\G^\ast$. The following definition is motivated by Theorem \ref{Boro2}.

\begin{definition}
	We say that a semisimple element $t \in (\mathbf{H}^\ast)^F$ is \textit{strictly quasi-isolated} if $\mathrm{C}_{(\mathbf{H}^\ast)^{F^\ast}}(s) \mathrm{C}_{\mathbf{H}^\ast}^\circ(s)$ is not contained in a proper Levi subgroup of $\mathbf{H}^\ast$.
\end{definition}

We say that a block $b$ of $\mathbf{H}^F$ is \textit{strictly quasi-isolated} if it occurs in $\mathcal{O} \mathbf{H}^F e_t^{\mathbf{H}^F}$ for a strictly quasi-isolated semisimple element $t \in (\mathbf{H}^\ast)^{F^\ast}$ of $\ell'$-order.

The following remark gives the generic structure of defect groups of blocks of groups of Lie type:

\begin{remark}
	Recall that for $\ell \geq 5$ (and $\ell \geq 7$ if $\G$ is of type $E_8$) we know by \cite[Lemma 4.16]{Marc} that every defect group $D$ of a block $b$ of $\mathbf{H}^F$ has a unique maximal abelian normal subgroup $Q$ such that the group extension
	$$1 \to Q \to D \to D/Q \to 1$$
	splits. One can now show that $\N_{\mathbf{H}^F}(Q)$ is always a proper subgroup of $\mathbf{H}^F$ unless $D$ is central in $\mathbf{H}^F$. Evidence suggests that working with $Q$ is more accessible than working with the defect group and the iAM-condition is usually checked by working with $\mathrm{N}_{\mathbf{H}^F}(Q)$.
\end{remark}

Following the terminology in \cite[Section 3.4]{Cabanesgroup} we say that an $\ell$-group $D$ is \textit{Cabanes} if it has a unique maximal abelian normal subgroup. This motivates the following hypothesis:

\begin{hypothesis}\label{assumptionquasi}
	Consider the class $\mathcal{H}_{\G}$ of pairs $( \mathbf{H},F')$ consisting of a simple algebraic group $\mathbf{H}$ of simply connected type over $\overline{\mathbb{F}_p}$ with Frobenius $F': \mathbf{H} \to \mathbf{H}$ such that the Dynkin diagram of $\mathbf{H}$ is isomorphic to a subgraph of the Dynkin diagram of $\G$. Assume that for $(\G,F)$ and $\ell$ one of the following holds:
	\begin{enumerate}[label=(\alph*)]
		\item We have $\ell \geq 5$ (and $\ell \geq 7$ if $\G$ is of type $E_8$).  Let $( \mathbf{H},F') \in \mathcal{H}_{\G}$ and $b$ a strictly quasi-isolated block of $\mathbf{H}^{F'}$. If $b$ has a non-central defect group $D$ and $\mathbf{H}^{F'}/ \mathrm{Z}(\mathbf{H}^{F'})$ is an abstract simple group then there exists an iAM-bijection $\Psi: \Irr_0(\mathbf{H}^{F'},b) \to \Irr_0(\mathrm{N}_{\mathbf{H}^{F'}}(Q), B_Q)$, where $Q$ is the maximal normal abelian subgroup of $D$.
		\item Let $( \mathbf{H},F') \in \mathcal{H}_{\G}$ and $b$ a strictly quasi-isolated block of $\mathbf{H}^{F'}$. If $b$ has a non-central defect group $D$ and $\mathbf{H}^{F'}/ \mathrm{Z}(\mathbf{H}^{F'})$ is an abstract simple group then there exists an iAM-bijection $\Psi: \Irr_0(\mathbf{H}^{F'},b) \to \Irr_0(\mathrm{N}_{\mathbf{H}^{F'}}(D), B_D)$.
	\end{enumerate}

\end{hypothesis}

\begin{remark}\label{solvable}
Note that in very small cases the group $\mathbf{H}^{F'}$ can be almost simple or solvable, see \cite[Theorem 24.17]{MT} for a list of these exceptions. In the case where $\mathbf{H}^{F'}$ is almost simple and occurs as a group in Hypothesis \ref{assumptionquasi} this list shows that necessarily $\mathbf{H}^{F'} \cong \mathrm{Sp}_4(2) \cong S_6$.
\end{remark}
%By the explicit description of automorphisms of $\G^F$ in \ref{automorphism}, every automorphism $\mathrm{Aut}(\G^F)$ lifts to a bijective automorphism of $\G$ commuting with the action of $F$.

%
%In these cases, the group $\G^F $ can have additional automorphisms which do not lift to bijective morphisms of $\G$. We are however not interested in these "non-geometric" automorphisms and will therefore exclude them in the following definition.

In the proof of Proposition \ref{directproduct} we will use the following property of Cabanes groups in an essential way (see \cite[Lemma 3.11]{Cabanesgroup} for a slightly more general statement):

\begin{lemma}\label{Cabanesgroup}
	Let $P:=P_1 \times P_2$, where $P_1$ and $P_2$ are both Cabanes. Then $P$ is Cabanes with maximal normal abelian subgroup $A_1 \times A_2$ where $A_i$, $i=1,2$, is the maximal normal abelian subgroup of $P_i$.
\end{lemma}

%\begin{proof}
%Let $A$ be a maximal abelian normal subgro
%\end{proof}

%In the following we abbreviate $L_0=[\Levi,\Levi]^F$.
%
%We show that the previous proposition applies to blocks covered by $c$.
We keep the notation of Theorem \ref{reductionquasi}.
We aim to understand blocks of $[\Levi,\Levi]^F$ which are covered by $c$. For this we need the following slight generalization of \cite[Proposition 2.3]{Bonnafe}:

\begin{lemma}\label{Bonnafe}
	Suppose that $\tilde{\mathbf{H}}$ and $\mathbf{H}$ are connected reductive groups with Frobenius endomorphism $F$. Let $\pi: \tilde{\mathbf{H}} \to \mathbf{H}$ be a surjective $F$-equivariant morphism with central kernel. If $\tilde{t} \in \tilde{\mathbf{H}}^F$ is strictly quasi-isolated  then its image $t:=\pi(\tilde{t})$ is strictly quasi-isolated as well.
\end{lemma}

\begin{proof}
	Let $\mathbf{M}$ be a Levi subgroup of $\mathbf{H}$ containing $\mathrm{C}_{\mathbf{H}^{F}}(t) \mathrm{C}^\circ_{\mathbf{H}}(t)$. Then $\tilde{\mathbf{M}}:=\pi^{-1}(\mathbf{M})$ is a Levi subgroup of $\tilde{\mathbf{H}}$. By \cite[(2.1),(2.2)]{Bonnafe} together with the $F$-equivariance of $\pi$ we deduce that $$\pi(\mathrm{C}_{\tilde{\mathbf{H}}^{F}}(\tilde{t}) \mathrm{C}^\circ_{\tilde{\mathbf{H}}}(\tilde{t}) ) \subset \mathrm{C}_{\mathbf{H}^{F}}(t) \mathrm{C}^\circ_{\mathbf{H}}(t) \subset \mathbf{M}.$$
From this it follows that $\mathrm{C}_{\tilde{\mathbf{H}}^{F}}(\tilde{t}) \mathrm{C}^\circ_{\tilde{\mathbf{H}}}(\tilde{t}) \subset \pi^{-1}(\mathbf{M})= \tilde{\mathbf{M}}$. Since $\tilde{t}$ is strictly quasi-isolated by assumption we deduce that $\tilde{\mathbf{M}}=\tilde{\mathbf{H}}$ and therefore $\mathbf{M}={\mathbf{H}}$. In other words, $t$ is strictly quasi-isolated.
\end{proof}

\begin{lemma}\label{commutator}
	Suppose that we are in the situation of Theorem \ref{reductionquasi}. Let $c_0$ be a block of $[\Levi,\Levi]^F$ covered by $c$. Then $c_0$ is a strictly quasi-isolated block of $[\Levi,\Levi]^F$.
\end{lemma}

\begin{proof}
	Since $\Levi=[\Levi,\Levi] \mathrm{Z}^\circ(\Levi)$ we deduce that the inclusion $\iota: [\Levi,\Levi] \hookrightarrow \Levi$ induces a dual morphism $\iota^\ast: \Levi^\ast \twoheadrightarrow [\Levi,\Levi]^\ast$ with central kernel. We let $\overline{s}:= \iota^\ast(s)$ be the image of $s$ under this map. Recall that $\Levi^\ast$ is the minimal Levi subgroup of $\G^\ast$ containing $\mathrm{C}_{(\G^\ast)^{F^\ast}}(s) \mathrm{C}^\circ_{\G^\ast}(s)$. Hence the element $s$ is strictly quasi-isolated in $\Levi^\ast$ and so $c$ is a strictly quasi-isolated block of $\Levi^F$. By Lemma \ref{Bonnafe} it follows that $\overline{s}$ is strictly quasi-isolated in $[\Levi, \Levi]^\ast$. Since $c$ is a block of $\mathcal{O} \Levi^F e_s^{\Levi^F}$ it follows that $c_0$ is a block of $\mathcal{O} [\Levi,\Levi]^F e_{\overline{s}}^{[\Levi,\Levi]^F}$. From this we conclude that $c_0$ is a strictly quasi-isolated block of $[\Levi,\Levi]^F$. 
\end{proof}

The proof of the following proposition is similar to the proof of \cite[Corollary 6.3]{JEMS}.

%\begin{definition}\label{geometricIAM}
%Let $\G$ be a connected reductive group with Frobenius endomorphism $F: \G \to \G$.
%
%\begin{enumerate}[label=(\alph*)]
%\item 
%
%We write $\mathrm{Aut}(\G,F)$ for the set of automorphisms of $\G^F$ which are obtained by restriction to $\G^F$ from bijective endomorphisms $\sigma: \G \to \G$ satisfying $\sigma \circ F= F \circ \sigma$.
%\item We say that $\Psi: \Irr_0(G , b) \to \Irr_0(M , B)$ is a geometric iAM-bijection for the block $b$ with respect to the subgroup $M \leq \G^F$ if the conditions in Definition \ref{IAMsuitable} hold for $\Gamma$ replaced by $\mathrm{Aut}(\G,F)$.
%\end{enumerate}
%\end{definition}
%
%Note that a similar notion of Definition \ref{geometricIAM}(a) was introduced in the comments following \cite[Remark 6.2]{Bonnafe2}.

\begin{proposition}\label{directproduct}
	Let $c_0$ be a strictly quasi-isolated block of $L_0:=[\Levi,\Levi]^F$ of non-central defect. If Hypothesis \ref{assumptionquasi} holds for $(\G,F)$ then there exists a defect group $D_0$ of $c_0$ and a characteristic subgroup $Q_0$ of $D_0$ satisfying $\mathrm{N}_{L_0}(Q_0) \lneq L_0$ and an iAM-bijection 
	$\varphi_0:  \Irr_0(L_0 , c_0) \to \Irr_0(\mathrm{N}_{L_0}(Q_0) , (C_0)_{Q_0})$.
\end{proposition}

\begin{proof}
	Since $\Levi$ is Levi subgroup of a simple algebraic group $\G$ of simply connected type it follows that $[\Levi,\Levi]$ is semisimple of simply connected type, see \cite[Proposition 12.4]{MT}. Thus, we have 
	$$[\Levi, \Levi]=\mathbf{H}_1 \times \dots \times \mathbf{H}_r,$$
	where the $\mathbf{H}_i$ are simple algebraic groups of simply connected type. We have a decomposition
	$$[\Levi,\Levi]^\ast= \mathbf{H}_1^\ast \times \mathbf{H}_2^\ast \times \dots \times \mathbf{H}_r^\ast$$
	into adjoint simple groups.
	
	The action of the Frobenius endomorphism $F$ induces a permutation $\pi$ on the set of simple components of $[\Levi,\Levi]$. We let $\pi=\pi_1 \dots \pi_t$ be the decomposition of this permutation into disjoint cycles. For $i=1, \dots,t$ choose $x_i \in \Pi_i$ in the support $\Pi_i$ of the permutation $\pi_i$ and let $n_i=|\Pi_i|$ be the length of $\pi_i$. For every $1 \leq i \leq t$ the inclusion map
	$$\mathbf{H}_{x_i} \hookrightarrow \prod_{x \in \Pi_i} \mathbf{H}_{x}$$
	induces isomorphisms between $\mathbf{H}_{x_i}^{F^{n_i}} $ and $ (\prod_{x \in \Pi_i} \mathbf{H}_{x})^F$.
	%and $\mathrm{Aut}(\mathbf{H}_{x_i}) \cong\mathrm{Aut}(\prod_{x \in \mathcal{O}_i} \mathbf{H}_{x})$.
	Consequently, we have $$L_0=[\Levi, \Levi]^F \cong \mathbf{H}^{F^{n_1}}_{x_1} \times \dots \times \mathbf{H}^{F^{n_t}}_{x_t}$$ and
	$$[\Levi^\ast, \Levi^\ast]^{F^\ast} \cong (\mathbf{H}^\ast_{x_1})^{{F^\ast}^{n_1}} \times \dots \times (\mathbf{H}^\ast_{x_t})^{{F^\ast}^{n_t}}$$
	in the dual group. There exists a semisimple element $\overline{s} \in [\Levi^\ast, \Levi^\ast]^{F^\ast}$ of $\ell'$-order such that $c_0$ is a block of $\mathcal{O} [\Levi,\Levi]^F e_{\underline{s}}^{[\Levi,\Levi]^F}$.
	
	Writing $\overline{s}=s_1  \dots s_t \in (\mathbf{H}^\ast_{x_1})^{{F^\ast}^{n_1}} \times \dots \times (\mathbf{H}^\ast_{x_t})^{{F^\ast}^{n_t}}$ with $s_i \in  (\mathbf{H}^\ast_{x_i})^{{F^\ast}^{n_i}}$ we obtain a decomposition $$e^{[\Levi,\Levi]^F}_{\overline{s}}=e_{s_1}^{\mathbf{H}_{x_1}^{F}} \otimes   \dots \otimes e_{s_t}^{\mathbf{H}_{x_t}^F}.$$
	In particular, the block $c_0$ can be written as $c_0=c_{x_1} \otimes  \dots \otimes c_{x_t}$ where the $c_{x_i}$ are blocks of $\mathbf{H}^{F^{n_i}}_{x_i}$. Note that the blocks $c_{x_i}$ are strictly quasi-isolated in $\mathbf{H}^{F^{n_i}}_{x_i}$ since $c_0$ is assumed to be strictly quasi-isolated. In the following we denote $H_{x_i}:=\mathbf{H}_{x_i}^{F^{n_i}}$.
	By possibly reordering the factors of $L_0$ we can assume that there exists some integer $v$ such that the factor $\mathbf{H}^{F^{n_i}}_{x_i}$ is quasi-simple if and only if $i \leq v$. Hence we can decompose $L_0$ as 
	$$L_0:=L_{0,\mathrm{simp}} \times L_{0,\mathrm{solv}},$$
	such that $L_{0,\mathrm{simp}}=H_{x_1} \times \dots \times H_{x_v}$ is a direct product of quasi-simple non-abelian finite groups and $L_{0,\mathrm{solv}}$ is a product of groups which are either solvable or almost simple as discussed in Remark \ref{solvable}. This induces a decomposition $c_0 = c_{0,\mathrm{simp}} \otimes c_{0, \mathrm{solv}}$ of blocks. 
	
	%Then we have a decomposition $D_0= D_{0,\mathrm{simp}} \times D_{0, \mathrm{solv}}$ such that $D_{0,\mathrm{simp}}$ is a defect group of $c_{0,\mathrm{simp}}$ and $D_{0,\mathrm{solv}}$ is a defect group of $c_{0,\mathrm{solv}}$.
	
	By Hypothesis \ref{assumptionquasi} there exist for each $i \leq v$ a characteristic subgroup $Q_i$ of a defect group $D_i$ of $c_i$ such that there exist an $\mathrm{N}_{\mathrm{Aut}(H_i)}(Q_i,C_i)$-equivariant iAM-bijection $\varphi_i: \Irr_0( H_{x_i}, c_{x_i} ) \to \Irr_0( \mathrm{N}_{H_{x_i}}(Q_i) , C_{x_i})$, where $C_{x_i}$ is the unique block in $\mathrm{Bl}(\mathrm{N}_{H_{x_i}}(Q_i) \mid D_i)$ with $C_{x_i}^{H_{x_i}}=c_{x_i}$. We let
	$$\{ x_1,\dots,x_v\}=A_1 \cup A_2 \cup \dots \cup A_u,$$
	be the partition such that $x_j,x_k \in A_i$ whenever $n_j=n_k$ and there exists a bijective morphism $\phi: \mathbf{H}_{x_j} \to \mathbf{H}_{x_k}$ commuting with the action of $F^{n_i}$ such that $\phi(c_{x_j})=c_{x_k}$. For each $i$ we fix a representative $x_{i_j} \in A_i$. We denote $y_i:=x_{ij}$ and $m_i:=n_{i_j}$.
	
	For a bijective morphism $\phi: [\Levi,\Levi] \to [\Levi,\Levi]$ commuting with $F$ it holds that $c_0$ is strictly quasi-isolated if and only if $\phi(c_0)$ is strictly quasi-isolated. Moreover, the conclusion of the proposition holds for $c_0$ if and only if it holds for $\phi(c_0)$. Hence, without loss of generality, we may conjugate the block $c_0$ by an element of $\mathrm{Aut}(L_{0,\mathrm{simp}})$ such that the block $c_{0,\mathrm{simp}}$ is of the form
	$$c_{0,\mathrm{simp}}=\bigotimes_{i=1}^u c^{\otimes |A_i|}_{y_i},$$
	where the $c_{y_i}$ are all distinct blocks. Therefore, the block stabilizer of $c_0$ satisfies
	$$\mathrm{N}_{\mathrm{Aut}(L_0)}(c_{0,\mathrm{simp}}) \cong \prod_{i=1}^u \mathrm{N}_{\mathrm{Aut}(H_{y_i})}(c_{y_i}) \wr S_{|A_i|},$$
	where $S_{|A_i|}$ denotes the symmetric group on $|A_i|$ letters.
	It follows that $D_{0,\mathrm{simp}}:=\prod_{i=1}^u D^{|A_i|}_{y_i}$ is a defect group of $c_{0,\mathrm{simp}}$. Moreover, $Q_{0,\mathrm{simp}}:= \prod_{i=1}^u Q^{|A_i|}_{y_i}$ is a characteristic subgroup of $D_{0,\mathrm{simp}}$ by Lemma \ref{Cabanesgroup}. We let $C_{0,\mathrm{simp}}$ be the unique block of $\mathrm{N}_{L_{0,\mathrm{simp}}}(Q_{0,\mathrm{simp}})$ with defect group $D_{0,\mathrm{simp}}$ satisfying $(C_{0,\mathrm{simp}})^{L_{0,\mathrm{simp}}}= c_{0, \mathrm{simp}}$. Define $H_{A_i}:= H_{y_i}^{|A_i|}$. For every $i$ we have
	$$\mathrm{N}_{H_{A_i}}(Q^{|A_i|}_{y_i}) \cong \mathrm{N}_{H_{A_{y_i}}}(Q_{y_i})^{|A_i|}.$$
	Thus, for each $i=1,\dots,u$ we obtain bijections 
	$$\varphi_{y_i}^{|A_i|}: \Irr_0(H_{A_i},  c_{y_i}^{\otimes |A_i|} ) \to \Irr_0(\mathrm{N}_{H_{A_i}}(Q_{y_i}^{|A_i|}) ,  C_{y_i}^{\otimes |A_i|}).$$
	We claim that these bijections are iAM-bijections. Every character $\chi \in \Irr_0(H_{A_i},  c_{y_i}^{\otimes |A_i|} )$ is $\mathrm{N}_{\mathrm{Aut}(H_{A_i})}({Q_{y_i}^{|A_i|},c_{y_i}^{\otimes|A_i|}})$-conjugate to a character $\prod_{i=1}^u \chi_i$ such that for every $i$ and $j$ we either have $\chi_i=\chi_j$ or $\chi_i$ and $\chi_j$ are not $\mathrm{Aut}(H_{A_i})$-conjugate. Since the relation $\geq_b$ is preserved by automorphisms we may assume that $\chi$ is of this form. Hence, $\chi= \prod_{l=1}^t \psi_k^{r_l}$, where the $\psi_l$ are all distinct characters and the $r_l$ are some integers. Therefore, we have
	$$\mathrm{Aut}(H_{A_i})_{\chi} \cong \prod_{l=1}^t \mathrm{Aut}(H_{y_i})_{\psi_l} \wr S_{r_l}.$$
	In other words, the stabilizer of $\chi$ is a direct product of wreath products.
	By \cite[Theorem 2.21]{LocalRep} and \cite[Theorem 4.6]{LocalRep} the relation $\geq_b$ is compatible with wreath products and by \cite[Theorem 2.18]{LocalRep} and  \cite[Theorem 4.6]{LocalRep} it is compatible with direct products. Hence we can conclude that the bijection $\varphi_{y_i}^{|A_i|}$ is an iAM-bijection.
	
	Using this it follows directly by \cite[Theorem 2.18]{LocalRep} and \cite[Theorem 4.6]{LocalRep} that the bijection
	$$\varphi_{0,\mathrm{simp}}:= \prod_{i=1}^u  \varphi_{x_{i_j}}^{|A_i|} : \Irr_0(L_{0,\mathrm{simp}}, c_{0,\mathrm{simp}} ) \to \Irr_0( \mathrm{N}_{L_{0,\mathrm{simp}}}(Q_{0,\mathrm{simp}}) , C_{0,\mathrm{simp}})$$
	is an iAM-bijection.
	
	Let $D_{0,\mathrm{solv}}$ be a defect group of $c_{0,\mathrm{solv}}$ and $C_{0,\mathrm{solv}}$ be the Brauer correspondent of $c_{0,\mathrm{solv}}$ in $\mathrm{N}_{L_{0,\mathrm{solv}}}(D_0)$. By \cite[Theorem 7.1]{JEMS} (and the fact that all simple groups involved in $L_{0,\mathrm{solv}}$ are AM-good, see Remark \ref{solvable} and \cite{Breuer}) there exists a strong iAM-bijection $$\varphi_{0,\mathrm{solv}}:\Irr_0(L_{0, \mathrm{solv}}, c_{0,\mathrm{solv}} ) \to \Irr_0( \mathrm{N}_{L_0}(D_{0,\mathrm{solv}}) , C_{0,\mathrm{solv}}).$$
	We note $D_0= D_{0,\mathrm{simp}} \times D_{0, \mathrm{solv}}$ is a defect group of the block $c_0$. If we are in case (a) of Hypothesis \ref{assumptionquasi} we have $\ell \geq 5$ and so we obtain that $D_{0, \mathrm{solv}}$ is abelian (the Sylow $\ell$-subgroups of $L_{0,\mathrm{solv}}$ are all abelian). Thus, the subgroup $Q_0:=Q_{0,\mathrm{simp}} \times D_{0, \mathrm{solv}}$ is a characteristic subgroup of $D_0$ by Lemma \ref{Cabanesgroup}. Since the image under an automorphism of a solvable (resp. quasi-simple) finite group is solvable (resp. quasi-simple), we obtain 
	$$\mathrm{Aut}(L_0) \cong \mathrm{Aut}(L_{0,\mathrm{simp}}) \times \mathrm{Aut}(L_{0,\mathrm{solv}}).$$
	Hence, by \cite[Theorem 2.18]{LocalRep} and \cite[Theorem 4.6]{LocalRep} we obtain an iAM-bijection
	$$\varphi_0:=\varphi_{0,\mathrm{simp}} \times \varphi_{0,\mathrm{solv}} : \Irr_0(L_0, c_{0}) \to \Irr_0( \mathrm{N}_{L_0}(Q_0) , C_0).$$
	It finally remains to show that $Q_0$ is non-central in $L_0$. By assumption the block $c_0$ of $L_0$ has non-central defect group $D_0=D_{0,\mathrm{simp}} \times D_{0,\mathrm{solv}}$. Hence, it follows that either $D_{0,\mathrm{simp}}$ is non-central in $L_{0,\mathrm{simp}}$ or $D_{0,\mathrm{solv}}$ is non-central in $L_{0,\mathrm{solv}}$. In the latter case it follows immediately that $Q_0$ is non-central in $L_0$. In the former case, we observe that there exists some $i$ with $1 \leq i \leq v$ such that $c_{x_i}$ has non-central defect group. In particular, $\mathrm{N}_{H_{x_i}}(Q_i) \subsetneq H_{x_i}$. From this we conclude that $Q_0$ is non-central in $L_0$.
	%such that
	%$$\Psi( \Irr_0(c_i  \mid \nu) ) \subseteq \Irr_0(C_i \mid \nu )$$ for every $\nu \in \Irr(\mathrm{Z}(G))$ and
	%$$( G/ Z \rtimes \Gamma_\chi, G/Z, \overline{\chi}) \geq_{b} ( \mathrm{N}_{G}(D)/Z \rtimes \Gamma_\chi, \mathrm{N}_G(D)/Z, \overline{\Psi(\chi)} ),$$
	%for every $ \chi \in \Irr_0(b)$ and $Z= \mathrm{ker}(\chi_{\mathrm{Z}(G)})$, where $\overline{\chi}$ and $\overline{\Psi(\chi)}$ lift to $\chi$ and $\Psi(\chi)$, respectively.
\end{proof}

%\begin{corollary}
%If Hypothesis \ref{assumptionquasi} holds then there exists an $\mathrm{N}_{\tilde{L} \mathcal{A}}(D_0,(C_0)_{D_0})$-equivariant bijection 
%$\varphi_0:  \Irr_0(L_1 , c) \to \Irr_0(M_{L_1} , (C_0)_{D_0})$ such that
%$$(( \tilde{L} \mathcal{A})_\chi, L_1, \chi) \geq_b ((\tilde{M}_{L_1} \hat{M}_{L_1})_{\varphi(\chi)},M_{L_1}, \varphi_0(\chi)).$$
%for every character $\chi \in \Irr_0(L,c)$. 
%\end{corollary}

\subsection{Application of character triples}

We keep the notation of Section \ref{sectionquasi}. Furthermore, we fix a block $c_0$ of $L_0=[\Levi,\Levi]^F$ covered by the block $c$ of $L$ with defect group $D_0$.  By \cite[Theorem 9.26]{NavarroBook} we can assume that the defect group $D$ of $c$ satisfies $D_0=D \cap L_0$. Our aim in this section is to obtain an iAM-bijection for the block $c$.

\begin{lemma}\label{normalsubgroup}
	If Hypothesis \ref{assumptionquasi} holds for $(\G,F)$ then there exists an $\mathrm{N}_{\tilde{L} \mathcal{A}}(Q_0,C_{Q_0})$-equivariant bijection 
	$\hat{\varphi}:  \Irr_0(L , c) \to \Irr_0( \mathrm{N}_L(Q_0) , C_{Q_0})$ such that
	$$(( \tilde{L} \mathcal{A})_\chi, L, \chi) \geq_b (\mathrm{N}_{\tilde{L} \mathcal{A}}(Q_0)_{\varphi(\chi)}, \mathrm{N}_{L}(Q_0), \varphi(\chi))$$
	for every character $\chi \in \Irr_0(L,c)$.
\end{lemma}

\begin{proof}
%	Denote by $(C_0)_{Q_0} \in \mathrm{Bl}(\mathrm{N}_{L_0}(Q_0) \mid D)$ the unique block which satisfies $((C_0)_{Q_0})^{L_0}=e$.
By Proposition \ref{directproduct}
%	and Lemma \ref{centrdef}
	there exists an iAM-bijection 
	$$\varphi_0:  \Irr_0(L_0 , c_0) \to \Irr_0(\mathrm{N}_{L_0}(Q_0) , (C_0)_{Q_0}).$$
	The grouop $(\tilde{L} \mathcal{A})_\chi$ acts on $L_0$ since $L_0$ is a characteristic subgroup of $L$.
%	 so the claim is clear. Assume therefore that the defect group $D_0$ is central in $L_0$. Then the group $(\tilde{L} \mathcal{A})_\chi$ acts on $L$ and stabilizes the block $c$. Thus, $( \tilde{L} \mathcal{A})_\chi$ stabilizes the defect group $D$ up to $L$-conjugation. In particular, $L_1=L_0 D$ is $(\tilde{L} \mathcal{A})_\chi$-stable.
	
	Hence, we can apply the Butterfly Theorem, see Theorem \ref{Butterfly}, and we conclude that the bijection $\varphi_0:  \Irr_0(L_0 , c_0) \to \Irr_0(\mathrm{N}_{L_0}(Q_0) , (C_0)_{Q_0})$
	satisfies
	$$(( \tilde{L} \mathcal{A})_\chi, L_0, \chi) \geq_b (\mathrm{N}_{\tilde{L} \mathcal{A}}(Q_0)_{\varphi_0(\chi)},\mathrm{N}_{L_0}(Q_0), \varphi_0(\chi))$$
	for every character $\chi \in \Irr_0(L_0,e)$.
	%ELABORATE!
	
	Now we apply Proposition \ref{JEMSproposition} in the case $H=\mathrm{N}_{\tilde{L} \mathcal{A}}(Q_0)$, $N=L_0$ and $J=L$.
	%G=\tilde{L} \mathcal{A}),H=\mathrm{N}_{\tilde{L} \mathcal{A}}(Q_0),N=L_0,J=L, J \cap H= \mathrm{N}_L(Q_0)
	%RECALL WHY YOU GET THE CORRECT BLOCK HERE
	We obtain an $\mathrm{N}_{\tilde{L} \mathcal{A}}(Q_0,C_{Q_0})$-equivariant bijection 
	$\hat{\varphi}:  \Irr_0(L , c) \to \Irr_0( \mathrm{N}_L(Q_0) , C_{Q_0})$ such that
	$$(( \tilde{L} \mathcal{A})_\chi, L, \chi) \geq_b (\mathrm{N}_{\tilde{L} \mathcal{A}}(Q_0)_{\varphi(\chi)}, \mathrm{N}_{L}(Q_0), \varphi(\chi))$$
	holds for every character $\chi \in \Irr_0(L,c)$.
%	Finally, note that 
%	$$\mathrm{N}_{\tilde{L} \mathcal{A}}(Q_0,C_{Q_0})=\mathrm{N}_{\tilde{L} \mathcal{A}}(D,C_{D}) \mathrm{N}_{L}(Q_0),$$
%	which proves that $\hat{\varphi}$ is $\mathrm{N}_{\tilde{L} \mathcal{A}}(Q_0,C_{Q_0})$-equivariant.
\end{proof}

\subsection{The Dade--Nagao--Glauberman correspondence and character triples}

Observe that the subgroup $Q_0$ may fail to be a characteristic subgroup of the defect group $D$. Therefore, we can not directly apply Theorem \ref{reductionquasi}. 

Let us assume that we are in the situation of Hypothesis \ref{assumptionquasi}(b).
% In other words, $\hat{Q}=D_0$.
In the following proposition we construct a suitable bijection between the height zero characters of $\mathrm{N}_L(D_0)$ and the height zero characters of $\mathrm{N}_L(D)$. To do this, we will reduce the construction of such a bijection to a case where the Dade--Nagao--Glauberman correspondence in the form of Lemma \ref{DGN} is applicable.
 
\begin{proposition}
	There exists a bijection $\sigma: \Irr_0(\mathrm{N}_L(D_0) \mid D) \to \Irr_0(\mathrm{N}_L(D) \mid D)$
	such that
	$$(\mathrm{N}_{\tilde{L} \mathcal{A}}(D_0)_{\chi}, \mathrm{N}_{L}(D_0), \chi ) \geq_b (\mathrm{N}_{\tilde{L} \mathcal{A}}(D)_{\sigma(\chi)}, \mathrm{N}_{L}(D), \sigma(\chi)).$$
	for all $\chi \in \Irr_0(\mathrm{N}_L(D_0) \mid D)$.
\end{proposition}

\begin{proof}
	Every character of $\Irr_0(\mathrm{N}_L(D_0) \mid D)$ (respectively $\Irr_0(\mathrm{N}_L(D) \mid D)$) lies over a linear character $\nu \in \mathrm{Lin}(D_0)$ by \cite[Proposition 2.5(a)]{JEMS}.
	
	Let $\psi \in \Irr_0(\mathrm{N}_L(D_0) \mid D)$ and $\nu \in \Irr(D_0 \mid \psi)$. We let $\psi_0 \in \Irr(\mathrm{N}_L(D_0)_\nu \mid \nu)$ be the Clifford correspondent of $\psi$. By \cite[Proposition 2.5 (e)]{JEMS} we deduce that $\mathrm{bl}(\psi_0)$ and $\mathrm{bl}(\psi)$ have a common defect group. Thus, there exists some $n \in \mathrm{N}_L(D_0)$ such that $\mathrm{bl}(\psi_0^n)$ has defect group $D$ and $\psi_0^n \in \Irr(\mathrm{N}_L(D_0)_{\nu^n} \mid \nu^n)$. Replacing $\nu$ by $\nu^n$ and $\psi_0$ by $\psi_0^n$ we can therefore assume that $\mathrm{bl}(\psi_0)$ has defect group $D$. In particular, we have $D \subset \mathrm{N}_L(D_0)_\nu$ and thus the character $\nu$ is $D$-stable.
	
	Let $\mathcal{X}$ be a set of representatives of the $\mathrm{N}_L(D)$-conjugacy classes of $D$-stable characters in $\mathrm{Lin}(D_0)$. By Clifford theory, in order to prove the proposition it suffices to construct for every for $\nu \in \mathcal{X}$ a bijection
	$$\Irr_0(\mathrm{N}_L(D_0) \mid D \mid \nu) \to \Irr_0(\mathrm{N}_L(D) \mid D \mid \nu)$$
	with the required properties. Here, $\Irr_0(\mathrm{N}_L(D_0) \mid D \mid \nu)$ (and similarly $\Irr_0(\mathrm{N}_L(D) \mid D \mid \nu)$) denotes the subset of characters of $\Irr_0(\mathrm{N}_L(D_0) \mid D)$ covering the character $\nu$.

	We have $\mathrm{Ker}(\nu) \lhd M:=\mathrm{N}_L(D_0)_\nu D$. Consider the natural projection map $\bar{}: M \to M/\mathrm{Ker}(\nu)$. Then we have $[\overline{M},\overline{D_0}]=1$ since the character $\nu$ is $M$-stable. In particular, $\overline{D_0}$ is a central subgroup of $\overline{M}$.
	%		 Let $K:=N_L(D_0)_\nu$.
	Denote $X:=\mathrm{N}_{\tilde L \mathcal{A}}(D_0)_{\nu}$ and $\overline{X}:=X/ \mathrm{ker}(\nu)$. We apply
	Lemma \ref{DGN}
	and obtain a bijection
	$$\Pi_{\overline{D}}:\Irr_0(\overline{M} \mid \overline{D}) \to \Irr_0(\mathrm{N}_{\overline{M}}(\overline{D}) \mid \overline{D})$$
	with $(\overline{X}_\chi, \overline{M}, \chi) \geq_b (\mathrm{N}_{\overline{X}}(D)_\chi, \mathrm{N}_{\overline{M}}(D), \Pi_{\overline{D}}( \chi ))$ for every $\chi \in \Irr_0(\overline{M} \mid \overline{D})$. Let $\overline{\nu}$ be the character of $\overline{D}_0$ which inflates to $\nu$. Observe that if $\chi$ lies above $\overline{\nu}$, then also the character $\Pi(\chi)$ lies above $\overline{\nu}$ by Lemma \ref{propertiesofcpartd}. Note that $\mathrm{N}_{\overline{M}}(\overline{D})=\overline{\mathrm{N}_L(D)_\nu}$. Therefore, inflation yields a bijection
	$$\Pi:\Irr_0(\mathrm{N}_L(D_0)_\nu \mid D \mid \nu) \to \Irr_0(\mathrm{N}_L(D)_\nu \mid D \mid \nu),$$
	which by \cite[Lemma 3.12]{JEMS} satisfies
	$(X_\chi, M, \chi) \geq_b (\mathrm{N}_X(D)_\chi, \mathrm{N}_M(D), \Pi( \chi ))$ for every $\chi \in \Irr_0(M , D)$.
	By Clifford theory we thus obtain a bijection
	$$\sigma:\Irr_0(\mathrm{N}_L(D_0) \mid D \mid \nu) \to \Irr_0(\mathrm{N}_L(D) \mid D \mid \nu)$$
	We now apply \cite[Theorem 3.14]{JEMS} to obtain
	$$(\mathrm{N}_{\tilde{L} \mathcal{A}}(D_0)_{\chi}, \mathrm{N}_{L}(D_0), \chi ) \geq_b (\mathrm{N}_{\tilde{L} \mathcal{A}}(D)_{\sigma(\chi)}, \mathrm{N}_{L}(D), \sigma(\chi)).$$
	Hence, the bijection $\sigma$ has the required properties.	
\end{proof}

\begin{lemma}\label{normalsubgroup2}
	If Hypothesis \ref{assumptionquasi}(b) holds for $(\G,F)$ then the $\mathrm{N}_{\tilde{L} \mathcal{A}}(D,C_{D})$-equivariant bijection 
	$\varphi:=\sigma \circ \hat{\varphi}:  \Irr_0(L , c) \to \Irr_0( \mathrm{N}_L(D) , C_{D})$ such that
	$$(( \tilde{L} \mathcal{A})_\chi, L, \chi) \geq_b (\mathrm{N}_{\tilde{L} \mathcal{A}}(D)_{\varphi(\chi)}, \mathrm{N}_{L}(D), \varphi(\chi))$$
	for every character $\chi \in \Irr_0(L,c)$.
\end{lemma}

\begin{proof}
	This follows from applying \cite[Lemma 3.8(a)]{JEMS} (i.e. using the transitivity of the relation ``$\geq_b$'').
\end{proof}

\subsection{Properties of Cabanes subgroups}

Let us now assume that we are in the situation of Hypothesis \ref{assumptionquasi}(a). In this case the subgroup $Q_0$ is by construction the maximal  normal abelian subgroup of the defect group $D_0$. This time, we want to compare the height zero characters of $\mathrm{N}_L(Q_0)$ to the height zero characters of $\mathrm{N}_L(Q)$ where $Q$ is the maximal abelian normal subgroup of $D$. Unfortunately, our considerations in the previous section can not be applied to our situation. Therefore, we assume the following additional hypothesis:

\begin{hypo*}\label{assumptionquasi2}
Assume Hypothesis \ref{assumptionquasi}. If we are in the case of Hypothesis \ref{assumptionquasi}(a) we assume additionally that $\mathrm{N}_{L}(Q)=\mathrm{N}_{L}(Q_0)$.
\end{hypo*}

This hypothesis is discussed in more depth in \cite[Section 11]{Jordan3}.
%
%In particular, it is shown there that the hypothesis is always satisfied whenever $\G$ is of type $A$ or $C$.
We show that the analog of Lemma \ref{normalsubgroup2} holds whenever Hypothesis \ref{assumptionquasi}(b)' is satisfied.

\begin{lemma}\label{normalsubgroup3}
	If Hypothesis \ref{assumptionquasi}(b)' holds for $(\G,F)$ then the bijection 
	$\hat{\varphi}:  \Irr_0(L , c) \to \Irr_0( \mathrm{N}_L(Q) , C_{Q})$ from Lemma \ref{normalsubgroup} is $\mathrm{N}_{\tilde{L} \mathcal{A}}(Q,C_{Q})$-equivariant and satisfies
	$$(( \tilde{L} \mathcal{A})_\chi, L, \chi) \geq_b (\mathrm{N}_{\tilde{L} \mathcal{A}}(Q)_{\varphi(\chi)}, \mathrm{N}_{L}(Q), \varphi(\chi))$$
	for every character $\chi \in \Irr_0(L,c)$.
\end{lemma}

\begin{proof}
%	By Lemma \ref{normalsubgroup} we obtain an $\mathrm{N}_{\tilde{L} \mathcal{A}}(Q_0,C_{Q_0})$-equivariant bijection 
%		$\hat{\varphi}:  \Irr_0(L , c) \to \Irr_0( \mathrm{N}_L(Q_0) , C_{Q_0})$ such that
%		$$(( \tilde{L} \mathcal{A})_\chi, L, \chi) \geq_b (\mathrm{N}_{\tilde{L} \mathcal{A}}(Q_0)_{\varphi(\chi)}, \mathrm{N}_{L}(Q_0), \varphi(\chi))$$
%		for every character $\chi \in \Irr_0(L,c)$.
		
		Since $\mathrm{N}_{L}(Q_0)=\mathrm{N}_{L}(Q)$ by assumption we have $C_{Q_0}=C_{Q}$. We clearly have $\mathrm{N}_{\tilde{L} \mathcal{A}}(Q)_{\varphi(\chi)} \leq \mathrm{N}_{\tilde{L} \mathcal{A}}(Q_0)_{\varphi(\chi)}$. Conversely, if $x \in \mathrm{N}_{\tilde{L} \mathcal{A}}(Q_0)_{\varphi(\chi)}$ then $x$ stabilizes the block $c$. Since $Q$ is a characteristic subgroup of the defect group $D$ of $c$ there exists some $l \in L$ such that $lx$ stabilizes $Q$. Therefore $lx$ stabilizes $Q_0=Q \cap L_0$. Consequently, $l \in \mathrm{N}_L(Q_0)=\mathrm{N}_L(Q)$ and so $x \in \mathrm{N}_{\tilde{L} \mathcal{A}}(Q_0)_{\varphi(\chi)}$. Therefore we indeed have
		$$(( \tilde{L} \mathcal{A})_\chi, L, \chi) \geq_b (\mathrm{N}_{\tilde{L} \mathcal{A}}(Q)_{\varphi(\chi)}, \mathrm{N}_{L}(Q), \varphi(\chi)).$$
		In a similar vain, one proves that $\mathrm{N}_{\tilde{L} \mathcal{A}}(Q,C_{Q})=\mathrm{N}_{\tilde{L} \mathcal{A}}(Q_0,C_{Q_0})$ and thus the bijection $\hat{\varphi}$ has all the required properties.
\end{proof}

\subsection{Jordan decomposition for the Alperin--McKay conjecture}

\begin{lemma}\label{Coro24}
	Suppose that Hypothesis \ref{assumptionquasi}' holds for $(\G,F)$. Then there exists a bijection $\tilde{\varphi}: \Irr(\tilde{L} \mid \Irr_0(c) ) \to \Irr(\mathrm{N}_{\tilde{L}}(Q) \mid \Irr_0( C_{Q}))$
	such that $\tilde{\varphi}$ together with the bijection $\varphi: \Irr_0(L , c) \to \Irr_0( \mathrm{N}_L(Q) , C_{Q})$ constructed in the proof of Lemma \ref{normalsubgroup} satisfy assumptions (i)-(iii) of Theorem \ref{reductionquasi}.
\end{lemma}

\begin{proof}
	Choose a transversal $\mathcal{T}$ of the characters in $\Irr_0(L,c)$ under $\mathrm{N}_{\tilde{L}}(c)$-conjugation.
	By Lemma \ref{normalsubgroup2} (respectively Lemma \ref{normalsubgroup3}) we obtain an $\mathrm{N}_{\tilde{L} \mathcal{A}}(Q,C_{Q})$-equivariant bijection 
	$\varphi:  \Irr_0(L , c) \to \Irr_0( \mathrm{N}_L(Q) , C_Q)$ such that
	$$(( \tilde{L} \mathcal{A})_\chi, L, \chi) \geq_b (\mathrm{N}_{\tilde{L} \mathcal{A}}(Q)_{\varphi(\chi)}, \mathrm{N}_{L}(Q), \varphi(\chi)),$$
	for every character $\chi \in \Irr_0(L,c)$. Hence, for every character $\chi \in \mathcal{T}$ we obtain by Theorem \ref{sigmabijections} a bijection $$\sigma_{\tilde{L}_\chi}^{(\chi)}: \Irr(\tilde{L}_\chi \mid \chi) \to \Irr(\mathrm{N}_{\tilde{L}}(Q)_\chi \mid \varphi(\chi)).$$
	By Clifford correspondence we then obtain a bijection $$\tilde{\varphi}_\chi:\Irr(\tilde{L} \mid \chi) \to \Irr(\tilde{M}_L \mid \varphi(\chi)).$$
	The disjoint union of the bijections $\tilde{\varphi}_\chi$, $\chi \in \mathcal{T}$, induces a bijective map
	$$\tilde{\varphi}: \Irr(\tilde{L} \mid \Irr_0(c) ) \to \Irr(\tilde{M}_L \mid \Irr_0(C_D)).$$
	By Lemma \ref{propertiesofc}, the bijection $\tilde{\varphi}$ is $\Irr( \tilde{M}_L / M_L) \rtimes \hat{M}_L$-equivariant. Together with Lemma \ref{propertiesofcpartd}, it follows that the bijections $\varphi$ and $\tilde{\varphi}$ satisfy assumption (i) of Theorem \ref{reductionquasi}. Moreover, Lemma \ref{propertiesofc} show that assumption (ii) in Theorem \ref{reductionquasi} is satisfied.
	%
	%$\mathrm{bl}(\psi)= \mathrm{bl}(\sigma_J(\psi))^J$
	
	Let $\chi \in \Irr_0(L,c)$ and $\tilde{\chi} \in \Irr(\tilde{L} \mid \chi)$. Let $\chi_0 \in \Irr(\tilde{L}_\chi \mid \chi)$ be the Clifford correspondent of $\tilde{\chi}$. By Definition \ref{blockisomorphism}(ii) we have  $$\mathrm{bl}(\chi_0)=\mathrm{bl}(\sigma_{\tilde{L}_{\chi}}(\chi_0))^{\tilde{L}_\chi}.$$
	By construction of the map $\tilde{\varphi}$, the character $\sigma_{\tilde{L}_{\chi}}(\chi_0) \in  \Irr(M_{\tilde{L}} \mid \varphi(\chi))$ is the Clifford correspondent of $\tilde{\varphi}(\tilde{\chi})$. This shows that assumption (iii) in Theorem \ref{reductionquasi} is satisfied.
	%such that it maps characters covering the character $\nu \in \Irr(\mathrm{Z}(\tilde{L}))$ to a character covering $\nu$.
\end{proof}

We can now prove our main theorem.

\begin{theorem}\label{maintheorem}
	Let $\G$ be a simple algebraic group of simply connected type with Frobenius $F: \G \to \G$. Suppose that $S:= \G^F / \mathrm{Z}(\G^F)$ is simple and $\G^F$ is its universal covering group. Let $b$ be a block of $\mathcal{O} \G^F e_s^{\G^F}$ for a semisimple element $s \in (\G^\ast)^{F^\ast}$ of $\ell'$-order. If Hypothesis \ref{assumptionquasi}' holds for the group $(\G,F)$ then every $\ell$-block $b$ of $\G^F$ is AM-good for $\ell$.
\end{theorem}

\begin{proof}
	As in Theorem \ref{reductionquasi} let $c \in \mathrm{Z}(\mathcal{O} L e_s^{L})$ be the block idempotent corresponding to $b$ under the Morita equivalence between $\mathcal{O} L e_s^{L}$ and $\mathcal{O} G e_s^{G}$ given by $H^{\mathrm{dim}(\Y_\U^\G)}_c( \Y_\U^\G, \mathcal{O}) e_s^L$. By Lemma \ref{Coro24} there exists a bijection $\tilde{\varphi}: \Irr(\tilde{L} \mid \Irr_0(c) ) \to \Irr(\mathrm{N}_{\tilde{L}}(Q) \mid \Irr_0( C_{Q}))$
	such that $\tilde{\varphi}$ together with the bijection $\varphi: \Irr_0(L , c) \to \Irr_0( \mathrm{N}_L(Q) , C_{Q})$ constructed in the proof of Lemma \ref{normalsubgroup} satisfy assumptions (i)-(iii) of Theorem \ref{reductionquasi}. Hence, by Theorem \ref{reductionquasi} the block $b$ is therefore AM-good for $\ell$.
	%By Theorem \ref{reductionquasi}, there exists an $\mathrm{N}_{\tilde{G} E}(Q,B_{Q})$-equivariant bijection $\Psi:\Irr_0(G,b) \to \Irr_0(\mathrm{N}_G(Q), B_{Q})$ such that
	%$$(( \tilde{G} E)_\chi, G, \chi) \geq_b ((\tilde{M} \mathrm{N}_{GE}(Q))_{\Psi(\chi)},\mathrm{N}_G(Q), \Psi(\chi)).$$
	%Let $\mathrm{Z}:= \mathrm{Z}(G) \cap \mathrm{Ker}(\chi)$. By Corollary \ref{central subgroup} we have
	%$$(( \tilde{G} E)_\chi /Z, G/Z, \overline{\chi}) \geq_b ((\tilde{M} \mathrm{N}_{G E}(D))_{\chi'} / Z,\mathrm{N}_G(Q), \overline{\chi'}).$$
	%Thus, by the Butterfly Theorem, see Theorem \ref{Butterfly}, the block $b$ is iAM-good.
	%If the block $c_0$ has non-central defect then $\mathrm{N}_{L_0}(Q) \lneq L_0$ which shows that $\mathrm{N}_G(Q) \lneq G$. Hence the block $b$ is iAM-good.
\end{proof}

\begin{remark}\label{strong versus weak}
	Observe that Hypothesis \ref{assumptionquasi} only requires that certain strictly quasi-isolated blocks have an iAM-bijection. In particular, if the $\ell$-block $b$ is already strictly quasi-isolated then the proof of the previous theorem shows that whenever $b$ has an iAM-bijection is already iAM-good. Indeed, using the proof of Lemma \ref{Coro24} and Theorem \ref{12} one can show that if $b$ is an (arbitrary) $\ell$-block of $\G^F$ which has an iAM-bijection then the block is automatically AM-good.
\end{remark}

\end{document}